\documentclass[a4paper,oneside,11pt]{amsart}

\usepackage{geometry}
\usepackage[english]{babel}
\usepackage{amsmath}
\usepackage[utf8]{inputenc}
\usepackage[T1]{fontenc}
 
\usepackage{amsthm}
\usepackage{amssymb}
\usepackage[all]{xy}
\usepackage[,colorlinks=true,citecolor= DarkBlue,linkcolor=Black]{hyperref}
\usepackage[svgnames,table]{xcolor}
\usepackage{graphicx}

\usepackage{pgf,tikz}\usepackage{mathrsfs}\usetikzlibrary{arrows}

\theoremstyle{definition}
\newtheorem{definition}{Definition}[section]

\theoremstyle{plain}
\newtheorem{theorem}{Theorem}
\newtheorem*{theorem*}{Theorem}
\newtheorem{proposition}[definition]{Proposition}
\newtheorem*{proposition*}{Proposition}
\newtheorem{lemma}[definition]{Lemma}
\newtheorem*{lemma*}{Lemma}

\newtheorem*{sublemma*}{Sub-lemma}
\newtheorem{corollary}[definition]{Corollary}

\newtheorem{conjecture}{Conjecture}

\newtheorem*{fact*}{Fact}

\newtheorem*{claim*}{Claim}

\theoremstyle{remark}
\newtheorem{remark}[definition]{Remark}

\newcommand\chech[1]{\mathaccent"7014{#1}}

\newcommand{\C}{\mathbf{C}}
\newcommand{\R}{\mathbf{R}}

\newcommand{\gl}{\mathfrak{gl}}
\renewcommand{\sl}{\mathfrak{sl}}

\newcommand{\heis}{\mathfrak{heis}}
\newcommand{\aff}{\mathfrak{aff}}

\renewcommand{\a}{\mathfrak{a}}
\newcommand{\g}{\mathfrak{g}}
\newcommand{\h}{\mathfrak{h}}    
\renewcommand{\k}{\mathfrak{k}}    
\newcommand{\m}{\mathfrak{m}}
\newcommand{\n}{\mathfrak{n}}
\newcommand{\so}{\mathfrak{so}}
\newcommand{\p}{\mathfrak{p}}
\newcommand{\gr}{\mathfrak{r}}
\newcommand{\s}{\mathfrak{s}}

\newcommand{\z}{\mathfrak{z}}

\renewcommand{\t}{\mathfrak{t}}

\renewcommand{\d}{\mathrm{d}}

\DeclareMathOperator{\Diff}{Diff}

\DeclareMathOperator{\Isom}{Isom}
\DeclareMathOperator{\Ad}{Ad}  
\DeclareMathOperator{\ad}{ad}

\DeclareMathOperator{\Span}{Span}
\DeclareMathOperator{\Aut}{Aut}

\DeclareMathOperator{\hol}{hol}

\DeclareMathOperator{\GL}{GL}

\DeclareMathOperator{\CO}{CO}

\DeclareMathOperator{\PO}{PO}

\DeclareMathOperator{\SO}{SO}

\DeclareMathOperator{\Ima}{Im}
\DeclareMathOperator{\Rea}{Re}

\DeclareMathOperator{\Rk}{Rk}
\DeclareMathOperator{\Conf}{Conf}

\DeclareMathOperator{\id}{id}
\DeclareMathOperator{\diag}{diag}

\DeclareMathOperator{\Heis}{Heis}

\renewcommand{\S}{\mathbf{S}}
\newcommand{\T}{\mathbf{T}}

\newcommand{\Ein}{\mathbf{Ein}}

\renewcommand{\epsilon}{\varepsilon}
\renewcommand{\geq}{\geqslant}
\renewcommand{\leq}{\leqslant}
\newcommand{\hx}{\hat{x}}

\renewcommand{\bar}{\overline}

\newcommand{\zn}{\mathfrak{z}(\mathfrak{n})}

\usepackage[percent]{overpic}

\title{Conformal actions of solvable Lie groups on closed Lorentzian manifolds}
\date{\today}
\author{Vincent Pecastaing}
\thanks{This work has been supported by the French government, through the UCA$^{\text{JEDI}}$ Investments in the Future project managed by the National Research Agency (ANR) with the reference number ANR-15-IDEX-01}
\address{Université Côte d’Azur, CNRS, LJAD, France}
\email{vincent.pecastaing@univ-cotedazur.fr}
\begin{document}

\maketitle

\begin{abstract}
We consider conformal actions of solvable Lie groups on closed Lorentzian manifolds.  With \cite{pecastaing_smooth}, in which we addressed similar questions for semi-simple Lie groups,  this work contributes to the understanding of the identity component $G$ of the conformal group of compact Lorentzian manifolds.  In the first part of the article, we prove that $G$ is inessential if and only if its nilradical is inessential.  In the second,  we assume the nilradical essential and establish conformal flatness of the metric on an open subset, under certain algebraic hypothesis on the solvable radical.  This is related to the Lorentzian Lichnerowicz conjecture.  Finally,  we consider the remaining situations where our methods do not apply to prove conformal flatness,  and conclude that for an essential compact Lorentzian $n$-manifold,  $n \geq 3$,  the radical of its conformal group admits a local embedding into $O(2,n)$.

\end{abstract}

\tableofcontents

\section{Introduction}

This work has two main motivations.  First,  we would like to obtain a classification,  up to local isomorphism, of conformal groups of compact Lorentzian manifolds,  that would generalize that of Adams-Stuck-Zeghib,  in which the identity component of their \textit{isometry} groups are classified.  Secondly,  the Lorentzian version of a conjecture of Lichnerowicz asserts that essential closed Lorentzian manifolds are conformally flat.  We aim to establish conformal flatness of these manifolds,  assuming that their conformal group satisfies extra algebraic assumption.  Let us first recall briefly these problems.

\subsection{General context}

A conformal diffeomorphism of a pseudo-Riemannian manifold $(M,g)$ is an element $f \in \Diff(M)$ for which there exists $\varphi \in \mathcal{C}^{\infty}(M)$,  $\varphi > 0$,  such that $f^*g = \varphi g$.  We denote by $[g]$ the conformal class of $g$ and we call $(M,[g])$ a pseudo-Riemannian conformal structure.  We also denote by $\Conf(M,[g])$ the group of conformal diffeomorphisms of $(M,[g])$.  

A natural question is to know if a conformal structure $(M,[g])$ admits conformal transformations which are \textit{not} isometries of any metric $g'$ in the conformal class.  

\begin{definition}
Let $(M,[g])$ be a  pseudo-Riemannian conformal structure.   A subgroup $H < \Conf(M,[g])$ is said to be \textit{inessential} if there exists $g' \in [g]$ such that $H <\Isom(M,g')$.  Otherwise,  it is said to be \textit{essential.}  If $\Conf(M,[g])$ is essential,  we say that $(M,[g])$ is essential.
\end{definition}

For positive-definite metrics,  a closed subgroup $H$ of conformal transformations is essential if and only if it acts non-properly on $M$.  So, for $M$ compact,  this is equivalent to the non-compactness of the conformal group.  These properties are no longer true for higher signatures,  so one of the first questions is how to characterize essentiality.

In dimension at least $3$,  a conformal class of pseudo-Riemannian metrics defines a \textit{rigid geometric structure},  be it in Gromov's sense \cite{gromov},  or in the sense that it is equivalent to a normalized \textit{Cartan geometry} modeled on $\Ein^{p,q}$,  the pseudo-Riemannian analogue of the Möbius sphere.  As a consequence,  $\Conf(M,[g])$ has the structure of a Lie transformation group (possibly non connected) and a natural problem is to determine which Lie groups arise in this way,  and on which geometry they can act.   We will concentrate here on conformal structures on compact manifolds and to the identity component of their conformal group.

Lichnerowicz conjectured in the 1960's that if a closed Riemannian manifold is essential,  then it is conformally equivalent to the round sphere.  More generally,  a ``Vague general conjecture'' of D'Ambra and Gromov (\cite{dambra_gromov}, §0.8) asserts that rigid geometric structures with large automorphism groups can be classified,  as it had been the case when Ferrand  and Obata  proved Lichnerowicz's conjecture,  first as formulated above,  and then its non-compact analogue \cite{ferrand71,ferrand96,obata71},  see also \cite{schoen} and the extension to the CR case.  Since then,  various people have been working on generalizations to other geometric structures, notably on a projective version of Lichnerowicz's conjecture (\textit{e.g.} \cite{matveev}).  It appeared that a pseudo-Riemannian analogue of Ferrand-Obata's theorem is not plausible \cite{frances_ferrand_obata_lorentz}.  One of the reasons is that it is deeply related to the ``rank $1$ nature'' of conformal Riemannian structures (\cite{frances_rang1}),  whereas other signatures are modeled on \textit{higher-rank parabolic geometries}.   In \cite{dambra_gromov},  the authors asked if general compact pseudo-Riemannian manifolds admitting an essential conformal group are conformally flat,  \textit{i.e.} locally conformally equivalent to the flat pseudo-Euclidean space $\R^{p,q}$.  Frances gave counter-examples to this question in all signatures of metric index at least $2$ (\cite{frances_counter_examples}).  The case of Lorentzian metrics remains however open,  and led to the following problem.

\begin{conjecture}
\label{conj:Lichnerowicz}
Let $(M,[g])$ be a compact,  conformal Lorentzian structure.  If it is essential, then it is conformally flat.
\end{conjecture}

Recently,  a proof for 3-dimensional,  real-analytic,  compact Lorentzian manifolds has been established in \cite{frances_melnick_lichne-anal3d},  and the main result of \cite{melnick_pecastaing} implies that it is also true for compact,  real-analytic Lorentzian manifolds with finite fundamental group.  

Apart from these general results,  various works have been dealing with compact pseudo-Riemannian manifolds with ``large'' conformal groups.   
It appeared that if a semi-simple Lie group $S$ of non-compact type acts conformally on  a closed pseudo-Riemannian manifold of signature $(p,q)$,  then $\Rk_{\R}(S) \leq \min(p,q)+1$ \cite{zimmer87} and if equality holds,  then the manifold is a quotient of the universal cover of $\Ein^{p,q}$ by a cyclic group \cite{bader_nevo,frances_zeghib}.  In \cite{pecastaing_rang1},  the opposite situation of actions by rank $1$ simple Lie groups is considered and similar conclusions are obtained when the metric index is minimal.  For the Lorentzian signature,  it is proved in \cite{pecastaing_smooth} that if a non-compact semi-simple Lie groups acts conformally essentially on a closed Lorentzian manifold,  then it is conformally flat (see §\ref{ss:intro_essentiality} below). 

In another direction,  it is natural to consider the ``solvable part'' of the conformal group of these manifolds and to derive geometric or dynamical information.  Frances and Melnick proved in \cite{frances_melnick_nilpotent} that if a nilpotent Lie group $N$, with nilpotence degree $k$, acts conformally on a closed pseudo-Riemannian manifold of signature $(p,q)$,  then $k \leq 2\min(p,q)+1$ and that if equality holds,  the manifold is again a quotient of the universal cover of $\Ein^{p,q}$ by a cyclic group.  Apart from this result,  to the best of our knowledge,  most results about solvable Lie group actions in pseudo-Riemannian geometry concern \textit{isometric actions} and not much is known about their conformal essential actions,  even in Lorentzian signature (see nonetheless \cite{belraouti_deffaf_raffed_zeghib} in the general homogeneous case and the announced results therein).  

\subsection{Embedding of the radical of the conformal group}

The present work focuses on the case of solvable Lie group actions on compact Lorentzian manifolds.  We obtain an optimal obstruction for an essential conformal action of a solvable Lie group on a closed Lorentzian manifold:

\begin{theorem}
\label{thm:embedding}
Let $(M^n,[g])$, $n \geq 3$, be a compact manifold endowed with a Lorentzian conformal structure, and let $R$ be a connected,  solvable Lie subgroup of $\Conf(M,[g])$.  If $R$ is essential,  then there exists a Lie algebra embedding $\gr \hookrightarrow \so(2,n)$. 
\end{theorem}
The  projective model of the Lorentzian Einstein Universe $\Ein^{1,n-1}$ is defined as the smooth quadric $\mathbb{P}(\{Q_{2,n}=0\}) \subset \R P^{n+1}$,  where $Q_{2,n}$ is a quadratic form of signature $(2,n)$.  It naturally comes with a conformally flat,  conformal class of Lorentzian metrics and its conformal group is $\PO(2,n)$.  So,  the converse statement is immediate: any Lie subgroup of $\PO(2,n)$ acts faithfully on (at least) one compact Lorentzian manifold. 

As a consequence of Liouville's theorem,  any Lie algebra of conformal vector fields of a conformally flat Lorentzian manifold embeds into $\so(2,n)$ (see also Lemma \ref{lem:iota_crochet}).  Therefore,  Theorem \ref{thm:embedding} supports Conjecture \ref{conj:Lichnerowicz},  because if true,  the latter would imply that the identity component of $\Conf(M,[g])$ locally embeds into $O(2,n)$.

Suppose that $\Conf(M,[g])_0$ is essential.  In \cite{pecastaing_smooth},  we proved that if the Levi factor of $\Conf(M,[g])_0$ is non-compact,  then the metric is conformally flat,  and in particular the whole identity component of the conformal group locally embeds into $O(2,n)$.  Applying Theorem \ref{thm:embedding} to the solvable radical $R$ of $\Conf(M,[g])_0$,  and Theorem \ref{thm:inessential_nilradical} below,  we conclude that for any closed Lorentzian manifold $(M,g)$,  if $\Conf(M,[g])_0$ is essential,  then,  up to local isomorphism,  it is a semi-direct product of a compact semi-simple Lie group with an immersed subgroup of $O(2,n)$.  

This raises the following question: Conversely,  which solvable Lie subgroups of $O(2,n)$ can be realized as the radical of the conformal group of some closed Lorentzian $n$-manifold? A similar problem was treated in \cite{adams_stuckII,zeghib98} in the classification of isometry groups of closed Lorentzian manifolds,  and turned out to be quite delicate.  We leave this question for further investigation ; it would ultimately be interesting to reach a statement comparable to Adams-Stuck-Zeghib's classification result,  that we recall below.

\begin{theorem*}[\cite{adams_stuck,adams_stuckII,zeghib98,zeghib_espaces_temps_homogenes}]
\label{thm:ASZ}
Let $(M,g)$ be a compact Lorentzian manifold and let $G  = \Isom(M,g)_0$ denote the identity component of its isometry group. Then,  its Lie algebra splits into a direct product $\g = \s \oplus \k \oplus \a$ where $\k$ is the Lie algebra of a compact semi-simple Lie group, $\a$ is abelian and $\s$ is in the following list: 
\begin{enumerate}
\item $\sl_2(\R)$
\item $\heis(2n+1)$, $n \geq 1$,
\item Oscillator algebras, \textit{i.e.} certain solvable extensions $\R \ltimes \heis(2n+1)$, $n \geq 1$,
\item $\{0\}$.
\end{enumerate}
Conversely,  for any such Lie algebra $\g$, there exists a compact Lorentzian manifold whose isometry group admits $\g$ as Lie algebra.
\end{theorem*}

\subsection{Essentiality of the nilradical}
\label{ss:intro_essentiality}

Let $(M,g)$ be a pseudo-Riemannian manifold.  Given an essential subgroup $G < \Conf(M,[g])$,  and a non relatively compact subgroup $H < G$,  a natural question is to know if $H$ is also essential.  Or equivalently: if $H$ preserves a metric in the conformal class,  is it also the case for all of $G$?

In our situation,  we obtained previously a positive answer when $H$ is semi-simple and non-compact

\begin{theorem*}[\cite{pecastaing_smooth}]
Let $(M,g)$ be a closed Lorentzian manifold and $S$ be a non-compact semi-simple Lie group. Suppose that $S$ acts conformally on $M$ with discrete kernel.
\begin{itemize}
\item If $S$ is inessential,  then so is all of $\Conf(M,g)_0$.
\item If $S$ is essential, then $(M,[g])$ is conformally flat.
\end{itemize}
\end{theorem*}

We would like to get a similar result in the setting of solvable Lie groups actions.  We prove here:

\begin{theorem}
\label{thm:inessential_nilradical}
Let $(M,g)$ be a closed Lorentzian manifold and let $R$ be a solvable Lie subgroup of $\Conf(M,[g])$.  Let $N$ be the nilradical of $R$.  If $N$ is inessential,  then so is $R$.
\end{theorem}

Furthermore,  when $N$ is non-abelian,  then its essentiality is characterized by that of $N_k$,  the last non-zero term of its lower-central series (see Theorem \ref{thm:nilpotent} below).

\begin{corollary}
\label{cor:inessential_nilradical}
Let $G$ the identity component of $\Conf(M,[g])$. Let $R \triangleleft G$ be its solvable radical and let $N \triangleleft \,  R$ be the nilradical.  If $G/R$ is compact,  then $G$ is essential if and only if $N$ is essential.
\end{corollary}

 Lorentzian manifolds for which $G/R$ is non-compact and $G$ is essential are conformally flat by \cite{pecastaing_smooth}.  The fact that their holonomy centralizes a non-compact simple Lie subgroup of $\PO(2,n) = \Conf(\Ein^{1,n-1})$ seems to be an indication that they are classifiable up to conformal equivalence,  justifying our assumption.  

We note  nonetheless that the statement is false in general when $G$ has a non-compact Levi factor.  For instance, if $(M,g)$ is a Hopf manifold,  $\Conf(M,g) \simeq O(1,n-1) \times \S^1$. The $\S^1$ factor is inessential but $\Conf(M,g)_0$ is not.  More generally,  we will observe:

\begin{proposition}
\label{prop:essential_noncompact_levi}
Let $G,R,N$ be as in Corollary \ref{cor:inessential_nilradical}.  If $N$ is inessential while $G$ is not,  then $R$ is abelian and $\g \simeq \s \oplus \mathfrak{r}$ as Lie algebras,  where $\s$ is a Levi factor of $\g$,  which necessarily contains a factor of non-compact type.
\end{proposition}

\subsection{Actions of nilpotent Lie groups}

By Corollary \ref{cor:inessential_nilradical},  if the identity component $G$ is essential and has compact Levi factor,  then some nilpotent subgroup of $G$ is also essential.    
As recalled above,  a general bound on the nilpotence degree of a nilpotent Lie group acting on closed pseudo-Riemannian manifolds,  as well as a geometric characterization of manifolds at the critical case,  was obtained in \cite{frances_melnick_nilpotent}.

We prove here the following for general nilpotent Lie group actions on closed Lorentzian manifolds.   For a nilpotent Lie algebra $\n$,  its nilpotence degree is the smallest integer $d \geq 1$ such that $\n_d = \{0\}$,  where we denote by $\{\n_i\}_{i \geq 1}$ the lower central series of $\n$.

\begin{theorem} 
\label{thm:nilpotent}
Let $H$ be a connected nilpotent real Lie group of nilpotence degree $k+1$ and let $(M^n,g)$, $n\geq 3$, be a compact Lorentzian manifold.  Let $H$ act locally faithfully by conformal transformations of $M$.  Then, we have the following.
\begin{enumerate}
\item Assume that $H$ is abelian.
\begin{enumerate}
\item Then $H$ acts locally freely on an open-dense subset of $M$,  hence $\dim H \leq n$.
\item If $H$ is essential,  then it admits either a fixed point,  or an isotropic $1$-dimensional orbit.
\item If $\dim H = n$ or $H \simeq \R^{n-1}$ and if $H$ acts faithfully and essentially,  then an open subset of $M$ is conformally flat.
\end{enumerate}
\item If $H$ is non-abelian,  then it is inessential if and only if $H_k$ acts locally freely.
\item If $H$ is non-abelian and essential,  then an open subset of $M$ is conformally flat.  Precisely, $\h$ has nilpotence degree $k \leq 3$, $\dim \h_k=1$ and if $X \in \h_k \setminus \{0\}$, then $X$ has a singularity of order $2$.
\end{enumerate}
\end{theorem}

\begin{remark}
For (1)(a) and (1)(b),  the results are true without assuming $M$ compact.
\end{remark}

\begin{remark}
We stress that for (1)(c),  our proof requires the action to be faithful in the case of $H=\R^{n-1}$.
\end{remark}

\begin{remark}
According to (1),  if $H$ is abelian of dimension greater than $1$,  and if $H$ acts locally freely,  then it is inessential.  The converse is false.

For $H = \R$,  consider the $3$-dimensional Hopf manifold $(M,[g]) = (\R^{1,2} \backslash \{0\}) / \langle 2 \id \rangle$.  Let $\{u^t\}$ be a unipotent one-parameter subgroup of $O(1,2)$ and let $\{k^t\}$ be the one induced by the homothetic flow on $\R^{1,2}$ (it factorizes into an $\S^1$-action). Then, the commutative product $\{u^tk^t\}$ is an essential conformal flow with no singularity.
\end{remark}

\begin{remark}
\label{rem:riemannian_abelian}
Point (1) generalizes easily to abelian conformal Lie group actions on Riemannian manifolds.  Essentiality is characterized by the existence of a global fixed point.  

However,  as shown in \cite{dambra} §.5,  there are examples of isometric actions of  $\R \times \mathbf{T}^k$ on closed pseudo Riemannian manifolds of metric index greater than $1$,  all of whose orbits are tori of dimension $k$.
\end{remark}

Combining Theorem \ref{thm:inessential_nilradical} and Theorem \ref{thm:nilpotent},  we get:

\begin{corollary}
Let $(M,g)$ be a closed real-analytic Lorentzian manifold and suppose that $G=\Conf(M,[g])_0$ is essential.  If the nilradical $N$ of $G$  is non-abelian,  then $(M,g)$ is conformally flat.
\end{corollary}

As explained above,  it is not clear in general if,  given an essential group $G$,  there exists a proper subgroup $H$ which is still essential.  In view of Conjecture \ref{conj:Lichnerowicz},  it is notably an important matter to know if $G$ contains an essential flow,  so that the assumption of the conjecture  can be replaced by that of an essential conformal vector field.  Note that this is the first step in the $3$-dimensional case \cite{frances_melnick_lichne-anal3d}.  

In this direction,  another consequence of what precedes is that in the real-analytic case,  if the identity component is essential,  then either the manifold is conformally flat or the nilradical is abelian and acts essentially.   

\subsection{Organization of the article}

In Section \ref{s:section2},  we introduce some materials that will be used frequently,  especially a version of Zimmer's embedding theorem for Cartan geometries,  which is a central ingredient in our methods.  

In Section \ref{s:inessential},  we start proving (1)(a),  (1)(b) and (2) in Theorem \ref{thm:nilpotent},  which relates inessentiality and locally free actions.  We then prove Theorem \ref{thm:inessential_nilradical},  first in the case where the nilradical is abelian and then in the non-abelian case.   
For these results,  the renormalization function is built using the pseudo-norm $g(X,X)$ of certain conformal vector fields $X$ associated to the action.  In the non-abelian case, Corollary \ref{cor:virtual_isotropy} gives precious information about the signature of the orbits,  which will guarantee that the renormalization function is everywhere positive.  We show that the corresponding conformal metric is $R$-invariant by combining elementary arguments based on the commutation relations between flows in the $R$-action.

Section \ref{s:essential} is devoted to essential actions of solvable Lie groups.  The results in the inessential case imply that the nilradical is essential.  In several contexts,  we show that either a flow or a non-trivial element in the nilradical admits a singularity $x$ which  of order $2$ (\textit{i.e.} the $1$-jet is trivial),  or more generally has a non-linear unipotent holonomy,  a notion that we recall at the beginning of Section \ref{s:essential},  with other tools derived from the normalized Cartan geometry associated to a conformal structure.  The results of \cite{frances_melnick_normal_forms} then imply that an open subset containing the singularity in its closure is conformally flat.  
The main result of \cite{bader_frances_melnick} plays a central role here.  Working in the closure of a light-like orbit of the nilradical $N$,  this result implies that the holonomy of vector fields in $\n$ have to satisfy certain bracket relations related to the adjoint representation of $\gr$ on $\n$,  when the latter has non purely imaginary eigenvalues.  With some technical algebraic considerations,  we will exhibit flows with non linear unipotent singularities in $N$ in every case that does not appear in Proposition \ref{prop:cas-defavorables}.  The section concludes with the case of conformal essential actions of abelian Lie groups of dimension $n$ and $n-1$,  concluding the proof of (1)(c) in Theorem \ref{thm:nilpotent}.  The adjoint action being trivial,  Zimmer's embedding theorem does not provide information. The fact that we are close to the critical dimension implies that the isotropy of a point in a light-like orbit is locally linearizable and conjugate to the full horospherical subgroup of $O(1,n-1)$ (if not, an open subset is conformally flat).  It follows that all light-like orbits are closed,  and an element of the form $\phi_X^{t_0}$,  $t_0>0$, admits a singularity of order $2$,  yielding conformal flatness of an open subset as before.

In the last section \ref{s:proof_embedding},  we conclude the proof of Theorem \ref{thm:embedding}.     The latter being immediate if we can prove that an open subset is conformally flat,  we can restrict to cases to which the methods of Section \ref{s:essential} do not apply.  These are described by the following.

\begin{proposition}
\label{prop:cas-defavorables}
Let $(M^n,g)$ be a closed Lorentzian manifold, with $n \geq 3$.  Suppose that a connected, solvable Lie group $R$ acts conformally and essentially on $M$.  Then,  either there exists a non-empty conformally flat open subset of $M$,  or  $\gr$ is isomorphic to a semi-direct product $\a \ltimes_{\rho} \n$,  where:
\begin{itemize}
\item $\n$ is  abelian,  with $1 \leq \dim \n \leq n-1$ ;
\item $\a$ is abelian and $\rho : \a \to \gl(\n)$ is a faithful representation such that $\rho(X)$ is semi-simple for all $X \in \a$;
\item Considering the complex weights $\{\alpha_1,\ldots,\alpha_r\} \subset (\a^*)^{\C}$ of $\rho^{\C}$,  if $\lambda_i = \Rea(\alpha_i)$ and $\mu_i = \Ima(\alpha_i)$,  then there exists $\lambda \in \a^*$ such that for all $i$,  $\lambda_i \in \{0,\lambda\}$.
\end{itemize}
\end{proposition}

With a last technical observation in the case $\dim \n = n-1$  (Lemma \ref{lem:dimN=n-1}),  it will then be easy to perform an explicit embedding of $\gr$ into $\so(2,n)$.

\subsection{Notations and conventions}

We will call $\heis$-triple in a Lie algebra any triple $(X,Y,Z)$ such that $[X,Y] = Z$,  $[X,Z] = [Y,Z]=0$ and $Z \neq 0$. They span a copy of the $3$-dimensional Heisenberg Lie algebra.  The lower central series $\{\g_k\}$ of a Lie algebra $\g$ is defined as $\g_0 = \g$ and $\g_{k+1} = [\g,\g_k]$.  For $\g$ nilpotent,  the nilpotence degree is defined as the smallest integer $k$ such that $\g_k = 0$.  

If $G$ is  a Lie group and $\h \subset \g$ a Lie subalgebra of its Lie algebra,  we call integral subgroup associated to  $\h$ the connected,  immersed subgroup of $G$ tangent to $\h$ at the identity.

If $G$ is a Lie group acting smoothly and with discrete kernel on a manifold $M$,  then we associate to any element $X \in \g$ the vector field $\bar{X}$ of $M$ which is the opposite of the infinitesimal generator of the flow corresponding to the action of $e^{tX}$.  The mapping $\{X \mapsto \bar{X}\}$ is then a Lie algebra embedding identifying $\g$ with a Lie algebra of vector fields of $M$,  which is implicitly used in all the article.

A Lie group action of a Lie group $G$ on a manifold $M$ is said to be locally faithful if its kernel is discrete. It is said to be locally free if the stabilizer of any point is discrete. We will use the same terminology for a Lie sub-algebra $\h \subset \g$ when the corresponding integral subgroup $H$ has the property in question.  For instance,  we will say ``$\h$ has a fixed point'' to mean that $H$ has a fixed point,  or ``$\h$ has a degenerate orbit at $x$''  to mean that the orbit $H.x$ is degenerate.  For a vector subspace $\h \subset \g$ and $x \in M$,  we will note $\h(x) = \{X_x, \ X \in \h\} \subset T_xM$.

If $X$ is a vector field of $M$,  we say that $X$ has a singularity of order $2$ at a point $x_0$  if $X(x_0)=0$ and $\d_{x_0}\phi_X^t = \id$,  \textit{i.e.} if its local flow has trivial $1$-jet at $x_0$.

If $(M,g)$ is a pseudo-Riemannian manifold on which a group $G$ acts conformally,  the conformal distortion of $G$ with respect to $g$ is the $\R_{>0}$-valued cocycle $\lambda : G \times M \to \R_{>0}$ defined for all $\varphi \in G$ and $x \in M$ by $[\varphi^* g]_x = \lambda(\varphi,x) g_x$.  The action of $G$ is said to be inessential if $g$ can be replaced by a conformal metric $g'$ with respect to which the action is isometric.

\section{Embedding theorem and other anterior results}
\label{s:section2}

\subsection{Embedding theorem for Cartan geometries}
\label{ss:embedding_general}

We will use several times a version of Zimmer's embedding theorem for Cartan geometries proved in \cite{bader_frances_melnick}, Theorem 4.1.  We recall the setting and its statement below.  We will apply a corollary of this theorem in Section \ref{s:inessential}, which can be stated (and in fact proved) without the formalism of Cartan geometries.  We will then use its full strength in Section \ref{s:essential} and afterwards.  We refer the reader to Section \ref{ss:cartanGeometryConformal} for a brief introduction to the Cartan connection associated to a conformal structure.  At first reading,   the reader can skip the general version and retain only Corollary \ref{cor:virtual_isotropy} below.

Let $H$ be a connected Lie group and let $S < H$ be a subgroup,  not required to be closed.  Following \cite{shalom}, define the \textit{discompact radical} of $S$ as the largest algebraic subgroup $\bar{S}_d$ in the  Zariski closure of $\Ad_{\h}(S)$ which does not admit any proper, algebraic, normal, cocompact subgroup. For instance, $\bar{H}_d = \Ad_{\h}(H)$ for $S=H$ an algebraic semi-simple Lie group of non-compact type.

Let $(M,\hat{M},\omega)$ be a Cartan geometry with an effective model space $(\g,P)$.  We note $\pi : \hat{M} \rightarrow M$ the fibration.  Let $H \to \Aut(M,\hat{M},\omega)$ be an action by automorphisms of the Cartan geometry,  which we assume to be a \textit{proper} Lie group homomorphism.  Any such action gives rise to a natural map $\iota  : \hat{M} \rightarrow \text{Mon}(\h,\g)$, where $\text{Mon}(\h,\g)$ denotes the variety of injective linear maps from $\h$ to $\g$,  defined by $\iota(\hx)(X) = \omega_{\hx}(X)$ for all $\hx \in \hat{M}$ and $X \in \h$.  Recall that by effectiveness of the model,  we can define without ambiguity lifts to the Cartan bundle of infinitesimal automorphisms of the Cartan geometry and $X$ can be seen as a vector field on $M$ or as a right-$P$-invariant vector field on $\hat{M}$.

\begin{theorem}[\cite{bader_frances_melnick}]
\label{thm:bfm}
Suppose that $\Ad_{\g}(P)$ is almost algebraic in $\GL(\g)$.  
If $S$ preserves a finite Borel measure $\mu$ on $M$,  then for $\mu$-almost every $x$,  for every $\hx \in \pi^{-1}(x)$, there exists an algebraic subgroup $\chech{S} < \Ad_{\g}(P)$ such that 
\begin{itemize}
\item for all $\chech{p} \in \chech{S}$, $\chech{p}. \iota_{\hx}(\h) = \iota_{\hx}(\h)$,
\item the induced homomorphism $\chech{S} \rightarrow \GL(\h)$ is algebraic, with image $\bar{S}_d$.
\end{itemize} 
\end{theorem}

Let $H$, $S$ and $\mu$ be as above.  For $x \in M$, we denote by $H_x$  the stabilizer of $x$ in $H$ and $\h_x$ its Lie algebra.  The tangent space $T_x(H.x)$ identifies with $\h / \h_x$,  we note $[q_x]$ the ray of quadratic forms on $\h/\h_x$ corresponding to the restriction of the conformal class $[g]$ to the tangent space of the $H$-orbit of $x$.  It is a general fact that the adjoint action of $H_x$ on $\h/\h_x$ is conformal with respect to $[q_x]$.  The following consequence of the previous theorem states that for $\mu$-almost every point,  the same holds for the discompact radical of $S$,  even when $H$ acts locally freely for instance. 

\begin{corollary}
\label{cor:virtual_isotropy}
For $\mu$-almost every point $x$,  $\h_x$ is $\bar{S}^d$-invariant and the induced action of $\bar{S}^d$ on $\h / \h_x$ is conformal with respect to $[q_x]$.
\end{corollary}

In particular,  when $S$ is amenable,  every compact $S$-invariant subset of $M$ contains a point where the conclusions of this corollary are true,  since amenability guaranties the existence of a finite $S$-invariant measure supported in this compact subset.

\subsection{Isotropic conformal vector fields}

We will use several times the following results (see for instance Lemma 6.1 of \cite{adams_stuck}).

\begin{lemma}
\label{lem:zero_kilingfield}
Let $X$ be a Killing vector field of a Lorentzian manifold $(M,g)$.  If $g(X,X) = 0$ and if $X$ vanishes at some point, then $X=0$.
\end{lemma}

\begin{lemma}[\cite{adams_stuck}, Lem.  6.5]
\label{lem:zero_nilpotentalgebra}
Let $M$ be a compact manifold and let $G$ be  a Lie group acting smoothly on $M$.  Let   $X,Y\in \g$ be such that $Z:=\ad(X)^k\,Y$ is centralized by $X$, for some $k \geq 1$.  If $Y$ vanishes at some point of $M$, then so does $Z$.
\end{lemma}

\begin{remark}
\label{rem:heis_localement_libre}
Note that if $(X,Y,Z)$ is an $\heis$-triple of vector fields of a compact manifold,  and if $Z$ has no singularity,  then the corresponding action of $\Heis(3)$ is locally free.
\end{remark}

Another tool that we will use is the following.

\begin{lemma}
\label{lem:collinearVectorFields}
Let $X,Y$ be two conformal vector fields of a pseudo-Riemannian manifold $(M,g)$.  If $X_x$ and $Y_x$ are collinear for every $x \in M$, then $X$ and $Y$ are collinear.  In particular,  in Lorentzian signature,  if $X$ and $Y$ are everywhere light-like and orthogonal, then they are proportional.
\end{lemma}

\begin{proof}
If $X \neq 0$, let $U \subset M$ be an open subset on which $X(x) \neq 0$. There is a smooth map $f$ defined on $U$ such that $Y = fX$ over $U$.  If we prove that $f$ is constant, then the lemma is established because $Y$ will coincide with a scalar multiple of $X$ on a non-trivial open subset of $M$.

So we are reduced to observe that for $X$ a non-vanishing conformal Killing vector field of $(M,g)$ and $f \in \mathcal{C}^{\infty}(M)$, if $fX$ is conformal, then $f$ is constant.  From the existence of $\phi,\psi \in \mathcal{C}^{\infty}(M)$ such that $\mathcal{L}_X g=\phi g$ and $\mathcal{L}_{fX}g=\psi g$, we deduce that for all vector fields $Y,Z$,
\begin{align*}
(\phi-f\psi)g(Y,Z) = (Y.f) g(X,Z) + (Z.f)g(X,Y).
\end{align*}
If we choose $Z \in X^{\perp}$ non collinear to $X$,  and if $Y \in Z^{\perp} \setminus X^{\perp}$,  then we get $Z.f = 0$. The same holds for all $Z \in X^{\perp}$ by density.  Now, because $X$ does not vanish, there exists an isotropic vector field $Y$ such that $g(X,Y) \neq 0$.  From the same relation with $Y=Z$,  we obtain $2(Y.f) g(X,Y) = 0$,  so $Y.f = 0$,  and because $Y \notin X^{\perp}$,  we deduce $\d f = 0$, as announced.
\end{proof}

\section{Inessential actions}
\label{s:inessential}

Let $H<\Conf(M,g)$ be a subgroup, which could be non-closed \textit{a priori}. Its conformal distortion with respect to $g$ is an $\R_{>0}$-valued cocycle $\lambda : H \times M \to \R_{>0}$ defined by $[\phi^* g]_x= \lambda(\phi,x) g_x$ for all $\phi \in H$,  $x \in M$.  The $H$-action is inessential if and only if there exists a function $\phi : M \to \R_{>0}$ such that $\lambda(h,x) = \phi(h.x)\phi(x)^{-1}$,  the invariant metric being $g/\phi$.

\subsection{Inessential nilpotent groups}
\label{ss:inessential_nilpotent}

We consider a connected nilpotent Lie group $N$ and a compact Lorentzian manifold $(M,g)$ on which $N$ acts conformally with discrete kernel.  We stress that the action is just assumed to be an immersion $N \to \Conf(M,[g])$ which could be non-proper.  We note $k+1$ the nilpotence degree of $\n$, \textit{i.e.} such that $\n_k$ is the last non-zero term in the lower central series of $\n$.  For all $i \leq k$ let $N_i$ denote the connected Lie sugroup of $N$ tangent to $\n_i$.

We start with an elementary proposition for the abelian case ($k=1$),  which is (1)(b) in Theorem \ref{thm:nilpotent}.  We stress that no compactness is needed.

\begin{proposition}
\label{prop:inessential_abelian}
Assume that $N \simeq \R^m$ is abelian. If all $N$-orbits have dimension greater than $1$, then $N$ is inessential. 
\end{proposition}

\begin{proof}
Let $\lambda : N \times M \rightarrow \R_{>0}$ be the conformal distortion of the $N$-action with respect to $g$. Pick $(X_1,\ldots,X_m)$ a basis of $\n$ and consider the function $\varphi$ defined on $M$ by 
\begin{align*}
\varphi(x) = \left ( \sum_{1\leq i,j \leq m} g_x(X_i,X_j)^2 \right )^{1/2}.
\end{align*}
By assumption, $\varphi$ does not vanish. So, it is a smooth positive function which satisfies $\varphi(g.x) = \lambda(g,x) \varphi(x)$ for every $g \in N$ because $N$ is abelian,  proving that $N$ is inessential.
\end{proof}

\begin{remark}
This observation generalizes immediately to pseudo-Riemannian signature $(p,q)$, assuming that orbits have dimension greater than $\min(p,q)$.
\end{remark}

\begin{remark}
\label{rem:spacelikeVectorFields}
Proposition \ref{prop:inessential_abelian} can be compared to a well known argument (see for instance \cite{obata70}, Theorem 2.4).  For $X$ a conformal vector field of a pseudo-Riemannian manifold $(M,g)$,  if $ g(X,X)$ does not vanish,  then the renormalized metric $g/|g(X,X)|$ is preserved by $X$.
\end{remark}

We now prove (1)(a) in Theorem \ref{thm:nilpotent}.  It follows that a conformal Lorentzian action of an abelian Lie group of rank at least $2$ is always inessential over an open-dense subset.

\begin{proposition}
\label{prop:locallyFreeOpenDense}
Let $N$ be an abelian Lie group acting conformally on a Lorentzian manifold $(M,g)$.  Then,  $N$ acts locally freely over an open-dense subset of $M$.
\end{proposition}

\begin{proof}
The approach follows the main argument \cite{dambra},  which was adapted in \cite{melnick_pecastaing}, §2.2, to the conformal setting,  but formulated in the context of a generalization of D'Ambra's theorem,  so non-valid in general. 

The map $x \in M \mapsto \dim N.x \in \mathbf{N}$ being lower semi-continuous,  there exists an open-dense subset $\Omega$ of $M$ on which it is locally constant.  Let $x \in \Omega$ and suppose to the contrary $\dim N.x < \dim N$.  Let $\{\phi^t\} \subset N_x$ be a one-parameter subgroup fixing $x$.  Since the $N$-orbits foliate trivially a neighborhood of $x$,  and $N$ being abelian,  it follows that $d_x \phi^t$ acts trivially on $T_x(N.x)$ and on $T_xM / T_x(N.x)$.  By Lemma 2.4 of \cite{melnick_pecastaing},  we obtain $\d_x \phi^t = \id$.  

Thus,  the flow $\phi^t$ has a singularity of order $2$ at $x$.  The proof of \cite{frances_melnick_normal_forms} shows that there are points $y$ arbitrarily close to $x$ such that $\phi^t(y) \to x$.   This contradicts the fact that a neighborhood of $x$ is trivially foliated by the $N$-orbits,  since,  replacing $y$ by $\phi^T(y)$ for $T$ large enough,  $y$ would have to belong to $N.x$,  the latter being fixed pointwise by $\phi^t$.
\end{proof}

For the non-abelian case,  we prove the following proposition (case (2) of Theorem \ref{thm:nilpotent}).

\begin{proposition}
\label{prop:inessential_nonabelian}
Assume that $N$ is non-abelian.  Then, $N$ is inessential if and only if $N_k$ acts locally freely everywhere.  
\end{proposition}
 
 \begin{remark}
 \label{rem:isometric_heis}
The fact that $N$ is inessential implies that $N_k$ acts locally freely is already known (for instance \cite{adams_stuck}, Corollary 6.3).  We point out that it follows independently of Lemma \ref{lem:Znull_translation} bellow: if the action is isometric,  then no non-trivial one-parameter subgroup has a trivial $1$-jet at a point,  hence the compact subset $K$ in question is empty.
\end{remark}  
 
 \begin{proof}
We start with the following intermediate result.
\begin{lemma}
\label{lem:heis3_triple_inessential}
Let $\h$ be a Lie algebra of conformal vector fields of $(M,g)$,  isomorphic to $\heis(3)$  and let $Z$ be a non-trivial element in the center.  If $\h$ acts locally freely,   then  $\h(x) \perp Z_x$ everywhere. In particular, $Z$ is everywhere isotropic and any $X \in \h \setminus \R.Z$ is space-like everywhere.	
\end{lemma}

\begin{proof}
Let $H'$ be the closure of the connected subgroup $H<\Conf(M,[g])$ corresponding to $\h$.  Then,  $H'$ is a $2$-step nilpotent Lie subgroup of $\Conf(M,[g])$ due to the following.   

\begin{lemma}
\label{lem:closure_nilpotent}
Let $G$ be a Lie group and $N<G$ an integral subgroup such that $\n$ is $2$-step nilpotent.  Then,  the closure $N'$ of $N$ in $G$ is a $2$-step nilpotent Lie subgroup of $G$. 
\end{lemma}

\begin{proof}
$N$ is $2$-step nilpotent as an abstract group, as can be easily seen using Baker-Campbell-Hausdorff formula.  Since $\mathcal{Z}(N) \subset \mathcal{Z}(N')$ and the latter is closed, it follows that $\bar{\mathcal{Z}(N)} \subset \mathcal{Z}(N')$.  Therefore,  the commutator $[x,y]$ of two elements of $N'$ must be central in $N'$,  because it is the limit of commutators in $N$, which are central in $N$.  It implies that $N'$ is itself $2$-step nilpotent, as an abstract group.  We are left to verify that a Lie group which is $2$-step nilpotent has a Lie algebra with the same property.  

Let $Z$ be the center of $N'$, and let $\z$ be its Lie algebra.  Fix $X,Y \in \n'$. For any $s,t \in \R$,  the BCH formula gives an explicit formula for the element $Z(s,t) \in \h$ such that 
\begin{align*}
e^{tX}e^{sY}e^{-tX}e^{-sY} = e^{s\Ad(e^{tX})Y} e^{-sY} = e^{Z(s,t)}.
\end{align*}
For $s,t$ sufficiently small,  $Z(s,t) \in \z$.  Consequently,  seen as an analytic map in $s$,  its first order term $\Ad(e^{tX})Y-Y$ belongs to $\z$.  Taking now derivative at $t=0$,  we obtain $[X,Y] \in \z$,  proving that $\n'$ is $2$-step nilpotent.
\end{proof}

Let $X \in \h$ be a non-central element and $S  = \{e^{tX}\}_{t\in \R} < H'$.  As $\h'$ is nilpotent,  the discompact radical of $S$ coincides with $\Ad_{\h'}(S)$.  
We apply Corollary \ref{cor:virtual_isotropy}  to this pair $(H',S)$: any closed, $S$-invariant  subset of $M$ contains a point $x$ such that $\Ad_{\h'}(S)$ preserves $\h_x'$ and $\Ad_{\h'}(S) \subset \Conf(\h'/\h_x',[q_x])$.  Since the action of $\Ad_{\h'}(S)$ on $\h '/ \h_x'$ is unimodular,  it is in fact isometric with respect to $q_x$.  Let $Y \in \h$ such that $Z = [X,Y]$.  We have $b_x(Z,Y) = b_x([X,Y],Y) = -b_x(Y,Z)$, so $b_x(Y,Z) = 0$, and similarly we obtain $b_x(Z,Z)=0$.  By density,  it follows that $Z$ is isotropic and orthogonal to the projection of $\h$ in $\h' / \h_x'$.  Consequently,  for any $X \in \h \setminus \z(\h)$,  $g_x(X,X)>0$.

Recall that this is true for some $x$ in the arbitrary compact $S$-invariant subset we picked.  Now consider $\{x \in M \ : \ g_x(X,X) \leq 0\}$.  It is of course compact,  but also $S$-invariant as $S$ acts conformally.  Therefore, this subset is empty and we obtain that $X$ is everywhere space-like as expected.   Thus,  we have shown that any non-central vector field of $\h$ is everywhere space-like.

As recalled in Remark \ref{rem:spacelikeVectorFields},  $X$ is inessential,  so that the conformal distortion of its flow has range in a compact sub-interval of $\R_{>0}$.  For $x \in M$,  we note $x_t = e^{tX}.x$.  Since $Z$ is central, $\lambda(e^{tX},x)  g_x(Z,Z)=g_{x^t}(Z,Z)$,  and similarly for $g(X,Z)$.  Consequently  either $g_x(Z,Z)=0$,  or the function $t \mapsto g_{x_t}(Z,Z)$ is bounded away from $0$, and similarly for $g(X,Z)$.  Pick now $Y \in \n$ such that $[X,Y]=Z$.  From the relation $\lambda(e^{tX},x) g_x(Y,Z)  = g_{x_t}(Y,Z) - t g_{x_t}(Z,Z)$,  we deduce $g_x(Z,Z)=0$,  proving that $Z$ is everywhere isotropic and $g_{x_t}(Y,Z)=\lambda(e^{tX},x) g_x(Y,Z)$ is either constant equal to $0$,  or bounded away from $0$.  Consequently,  using that $\lambda(e^{tX},x)g_x(Y,Y) = g_{x_t}(Y,Y)-2tg_{x_t}(Y,Z)$ for all $x \in M$ and $t \in \R$,  we get that $g_x(Y,Z) = 0$.  This conclusion is valid for any element $Y$ which does not centralize $X$,  so it must be valid for all elements of $\h$ by continuity.
\end{proof}

Let us observe now that $\dim N_k = 1$ and that its orbits are $1$-dimensional and isotropic.  For any $X \in \n_{k-1}$,  if $Y \in \n$ is chosen such that $Z:=[X,Y] \neq 0$,  then $(X,Y,Z)$ span a Lie algebra to which we can apply the previous lemma.  Indeed,  $Z \in \n_k$ is non-zero so must be central and nowhere vanishing by hypothesis and we can apply	 Lemma \ref{lem:zero_nilpotentalgebra}.  In particular,  for any $X\in \n_{k-1}$ and $Y \in \n$,  the vector field $[X,Y]$ is everywhere isotropic.  

\begin{lemma}
\label{lem:nk_dim1}
If,  for all $X \in \n_{k-1}$ and $Y \in \n$,   the vector field $[X,Y]$ is nowhere vanishing and isotropic,  then $\dim \n_k = 1$.
\end{lemma}

\begin{proof}
For any $x \in M$ and $X \in \n_{k-1}$,  the linear map $Z \in [X,\n] \mapsto Z_x \in T_xM$ is injective and its range is totally isotropic.  It follows that $\dim [X,\n] \leq 1$ for all $X \in \n_{k-1}$.  

Now, let $Z_1 = [X_1,Y_1]$ and $Z_2 = [X_2,Y_2]$ with $X_i \in \n_{k-1}$ and $Y_i \in \n$.  We distinguish two cases:
\begin{enumerate}
\item $[X_1,Y_2]=[X_2,Y_1]=0$.  In this case,  $Z_1 +Z_2 = [X_1+X_2,Y_1+Y_2] \in [X_1+X_2,\n]$ is everywhere isotropic.  The same being true for $Z_1$ and $Z_2$,  it implies that they are proportional by Lemma \ref{lem:collinearVectorFields}.
\item Either $[X_1,Y_2] \neq 0$ or $[X_2,Y_1] \neq 0$.  Let us assume $[X_1,Y_2] \neq 0$.  Because $\dim [X_1,\n] \leq 1$,  there exists $\lambda \in \R$ such that $Z_1 = [X_1,Y_1] = \lambda[X_1,Y_2]$. Therefore, $Z_1+Z_2=[\lambda X_1 + X_2,Y_2] \in [\lambda X_1 + X_2,\n]$.  Similarly to the first case,  we deduce that $Z_1$ is collinear to $Z_2$.
\end{enumerate}
We deduce that $\n_k=[\n_{k-1},\n]$ is a line as claimed. 
\end{proof}

By hypothesis,  the flow corresponding to $\n_k$ acts with no singularity on $M$ and has isotropic orbits by Lemma \ref{lem:heis3_triple_inessential}.
 
Finally,  let us prove that $N$ is inessential.  Let $X \in \n_{k-1} \setminus \z(\n)$.  Let $Y \in \n$ which does not commute with $X$ and let $Z = [X,Y] \in \n_k$.  Then,  $X,Y,Z$ span a Lie algebra $\h$ isomorphic to $\heis(3)$ and satisfying the hypothesis of Lemma \ref{lem:heis3_triple_inessential}. In particular, $g(X,X) > 0$ everywhere.  Let $T \in \n$ and let us note $Z' := [T,X]$. 
We distinguish two cases: 
\begin{enumerate}
\item If $Z'=0$,  then $\lambda(e^{tT},x) g_x(X,X)=g_{\phi_T^t(x)}(X,X)$ for all $x \in M$ and $t\in \R$.  
\item  If $Z' \neq 0$, then $(X,T,Z')$ is another $\heis(3)$-triple to which we can apply Lemma \ref{lem:heis3_triple_inessential}. Therefore, $g(Z',Z') = 0$ and $g(X,Z') = 0$ everywhere.  From $(\phi_T^t)_* X_x = X_{\phi_T^t(x)} - t Z'_{\phi_T^t(x)}$,  we then obtain $\lambda(e^{tT},x) g_x(X,X) = g_{\phi_T^t(x)}(X,X)$ as in the first case.  
\end{enumerate}
As the function $\varphi  = g(X,X)$ is everywhere positive,  we get that the conformal distortion of any one-parameter subgroup in $N$ is a  coboundary defined by this map.  So,  all of $N$ preserves $g/\varphi$ by connecteness. 
 \end{proof}

\subsection{Proof of Corollary \ref{cor:inessential_nilradical}}
\label{ss:proof_inessentialNilradical}

We prove in this section that Theorem \ref{thm:inessential_nilradical} implies Corollary \ref{cor:inessential_nilradical} by a standard averaging argument.  We will use the same notations $G$,  $R$ and $N$ and we assume that $G/R$ is compact and  that the nilradical $N $ is inessential.  By Theorem \ref{thm:inessential_nilradical},  $R$ is inessential.  

Consider $K<G$ a compact Levi factor, \textit{i.e.}  a compact semi-simple Lie group such that $\Conf(M,g)_0 = KR$ and $K \cap R$ is finite.  Observe that $K$ preserves the set of $R$-invariant metrics in the conformal class because it normalizes $R$.  Hence,  by averaging in this convex set of metrics, we obtain a conformal metric which is both $K$-invariant and $R$-invariant. 

\subsection{Inessential abelian nilradical}
\label{ss:inessential_abelian}

We prove in this section Theorem \ref{thm:inessential_nilradical} in the case where the nilradical $N \! \triangleleft  R$ is abelian.  If $N=R$,  the statement is obvious,  so we assume that $\gr \neq \n$ is not nilpotent.  We prove that $\gr \simeq \aff(\R) \oplus \R^k$ and that $R$ is inessential.  

Since $\n$ is abelian and $[\gr,\gr] \subset \n$,  if $\rho  : \gr \to \gl(\n)$ denotes the representation $\rho(X) = \ad(X)|_{\n}$, then $\rho$ has abelian image.  

\begin{lemma}
For all $X \in \gr \setminus \n$, we have $\dim [X,\n]=1$.
\end{lemma}

\begin{proof}
We first note that $\n = \ker \rho$ because $\R.X + \n$ is nilpotent if and only if $X \in \n$ by definition of the nilradical.  So, if $X \notin \n$, there exists $Y \in \n$ such that $Z = [X,Y]\neq 0$.  Since $(\phi_Y^t)_* X_x= X_{\phi_Y^t(x)} + tZ_{\phi_Y^t(x)} = X_{\phi_Y^t(x)} + t(\phi_Y^t)_*Z_x $ and $[Y,Z]=0$, we obtain for all $x \in M$ and $t \in \R$
\begin{align*}
g_x(X,Z) = g_{\phi_Y^t(x)}(X,Z) + t g_x(Z,Z),
\end{align*}
proving that $g_x(Z,Z)=0$.  Since $Z$ is moreover a Killing vector field of $(M,g)$,  Lemma \ref{lem:zero_kilingfield} implies that it does not vanish anywhere.  As a consequence,  for any $x \in M$, the map $Z \in [X,\n] \mapsto Z_x \in T_xM$ is injective and has isotropic image,  and we get $[X,\n]$ has dimension $1$.
\end{proof}

It follows that for $X \in \gr \setminus \n$, all eigenvalues of $\rho(X)^{\C} \in \gl(\n^{\C})$ are real because the existence of a non-real one would imply that of a real plane in the image of $\rho(X)$.  It also has exactly one non-zero eigenvalue,  and the corresponding eigenspace is $1$-dimensional, because if all eigenvalues were zero, then $\rho(X)$ would be nilpotent, implying $X \in \n$.  Since $\rho(X)$ has rank $1$,  it follows that it is diagonalisable over $\R$. Using that the $\rho(X)$, $X \in \gr$, form a commutative family,  we deduce that $\gr \simeq \aff(\R) \oplus \R^k$ for $k \geq 0$.

\begin{lemma}
\label{lem:affine_inessential}
Let $X,Y$ be non-zero conformal vector fields of $(M,g)$, with $[X,Y] = Y$ and $Y$ inessential. Then the Lie algebra they span is inessential, and $g(X,X) > 0$ everywhere.
\end{lemma}

\begin{proof}
Using that $Y$ is a Killing vector field of some conformal metric, which we may assume to be $g$,  we obtain that for all $x \in M$,  the function $t \mapsto g_{\phi_Y^t(x)}(Y,Y)$ is constant.  Because $\Ad(e^{tY}).X =  X - tY$,  we also have $g_x(X,Y) = g_{\phi_Y^t(x)}(X,Y) + t g_{\phi_Y^t(x)}(Y,Y)$,  showing that $g(Y,Y) = 0$ identically and $t \mapsto g_{\phi_Y^t(x)}(X,Y)$ is constant.  Similarly,  $g_x(X,X) = g_{\phi_Y^t(x)}(X,X)+2t g_{\phi_Y^t(x)}(X,Y)$ also implies $g(X,Y)=0$.  By Lemma \ref{lem:zero_kilingfield}, it follows that $Y$ does not vanish. Therefore, $\Span(X_x,Y_x)$ is $2$-dimensional everywhere because if $Z \in \Span(X,Y)$ vanishes at a point $x$,  then $[Z,Y]=\alpha Y$ for $\alpha \neq 0$,  and then $0 = (\phi_Y^t)_* Z_x = Z_{\phi_Y^t(x)} + \alpha t Y_{\phi_Y^t(x)}$ for all $t \in \R$, contradicting the fact that $Z$ is bounded and $Y$ never vanishes.

Therefore,  $g(X,X) > 0$ everywhere and since $g_x(X,X) = g_{\phi_Y^t(x)}(X+tY,X+tY) = g_{\phi_Y^t(x)}(X,X)$,  we obtain that $g/g(X,X)$ is preserved by both $X$ and $Y$.
\end{proof}

Finally,  when $\gr \neq \n$,  any $\aff(\R)$ factor in $\gr$ satisfies the hypothesis of this lemma.  If we choose any $X \in \gr \setminus \n$,  then we have $g(X,X)>0$ everywhere and $g/g(X,X)$ is $R$-invariant,  proving Theorem \ref{thm:inessential_nilradical} in the case where $N$ is abelian.

\subsection{Inessential non-abelian nilradical}
\label{ss:inessential_nonAbelian}

In this section, we treat the case of a non-abelian nilradical for the proof of Theorem \ref{thm:inessential_nilradical}. 

Let us assume that the nilradical $N \triangleleft \, R$ preserves a metric $g$.  We denote by $\lambda : R \times M \rightarrow \R_{>0}$ the conformal distortion with respect to $g$. Let $\z = [\n,\n]$ and let $\zn$ be the center of $\n$. According to Proposition \ref{prop:inessential_nonabelian} and Remark \ref{rem:isometric_heis}, $\z \subset \zn$ is a line.
We fix $Z \in \z$ a non-zero element.

\begin{lemma}
\label{lem:orthogonality_relations_loc_free}
For all $X \in \n$, $g(X,Z)=0$  and for $X \in \n \setminus \z(\n)$,  $g(X,X)>0$.
\end{lemma}

\begin{proof}
Let $X \in \n$.  If $X \notin \zn$, there is $Y\in \n$ such that $[X,Y]=Z$ and $(X,Y,Z)$ is an inessential $\heis(3)$-triple.  By Remark \ref{rem:isometric_heis}, the corresponding action of $\Heis(3)$ is everywhere locally free and $g(X,Z) = 0$ and $g(X,X)>0$ everywhere.  If $X \in \zn$ and if $X',Y \in \n$ are such that $[X',Y] = Z$, then we also have $[X+X',Y]=Z$. So, $g(X+X',Z) = g(X',Z)=0$ and we get $g(X,Z)=0$.
\end{proof}

Note that $\z$ and $\zn$ are also ideals of $\gr$.  

\begin{lemma}
\label{lem:[g,n]cap zn}
We have $[\gr,\n] \cap \zn \subset \z$. In particular, $[\gr,\zn] \subset \z$. 
\end{lemma}

\begin{proof}
Let $X \in \gr$ and $X_0 \in \n$. Assume that $X_1 = [X,X_0] \in \zn$. Let $x \in M$. Because $X_0$ and $X_1$ commute and $N$ is isometric with respect to $g$, we have $g_{\phi_{X_0}^t(x)}(X_1,X_1) = g_x(X_1,X_1)$ for all $t \in \R$.  Using $\Ad(e^{tX_0})X = X - t X_1$ and the corresponding action of the flow of $X_0$ on $X$,  we get
\begin{align*}
g_x(X,X) = g_{\phi_{X_0}^t(x)}(X,X) - 2t g_{\phi_{X_0}^t(x)}(X,X_1) + t^2 g_x(X_1,X_1).
\end{align*}
It follows that $g_x(X_1,X_1) = 0$ everywhere.  Since we also have $g(X_1,Z) = 0$ by Lemma \ref{lem:orthogonality_relations_loc_free},
 it means that $(X_1)_x$ is collinear to $Z_x$ for all $x \in M$.  Thus, for some $\lambda \in \R$, $X_1 - \lambda Z$ is a light-like Killing vector field of $(M,g)$ which vanishes at some point. According to Lemma  \ref{lem:zero_kilingfield},  it follows that $X_1-\lambda Z = 0$, \textit{i.e.} $X_1 \in \z$ as claimed.
\end{proof}

It is a general fact that $[\gr,\gr]$ is a nilpotent ideal of $\gr$, so $[\gr,\gr] \subset \n$ and $\gr / \n$ is abelian.  Consider the representation induced by the adjoint representation
\begin{align*}
 \rho : \gr / \n \rightarrow \gl(\n/\zn).
 \end{align*}

\begin{lemma}
\label{lem:purely_imaginary}
For all $\bar{X} \in \gr / \n$,  all the eigenvalues of $\rho(\bar{X})^{\C}$ are purely imaginary.
\end{lemma}

\begin{proof}
Assume to the contrary that $1+i\theta$ is a complex eigenvalue of some $\rho(\bar{X})^{\C}$.  It implies that there exist non-zero vectors $\bar{X_0}$ and $\bar{Y}_0$ in $\n / \zn$ such that 
\begin{align*}
\begin{cases}
\rho(\bar{X}) \bar{X_0} = \bar{X_0} + \theta \bar{Y_0} \\
\rho(\bar{X}) \bar{Y_0} = \bar{Y_0} - \theta \bar{X_0} .
\end{cases}
\end{align*}
Assume first that $\theta \neq 0$ so that $\bar{X_0}$ and $\bar{Y_0}$ are linearly independent.  Now, for arbitrary $X \in \gr$ and $X_0,Y_0 \in \n$ projecting to $\bar{X}$,  $\bar{X_0}$ and $\bar{Y_0}$ respectively, we have $X_1,Y_1 \in \zn$ such that
\begin{align*}
\begin{cases}
[X,X_0] = X_0 + \theta Y_0 + X_1 \\
[X,Y_0] = Y_0 + \theta X_0 + Y_1.
\end{cases}
\end{align*}

Let us define $X_0'=  X_0 + \frac{1}{1+\theta^2}X_1 - \frac{\theta }{1+\theta^2} Y_1$  and $Y_0' = Y_0 + \frac{\theta }{1+\theta^2} X_1+ \frac{1}{1+\theta^2} Y_1$. Since $[\gr,\zn] \subset \z$ (Lemma \ref{lem:[g,n]cap zn}), we obtain :
\begin{align}
\label{eq:similitude}
\begin{cases}
[X,X_0'] = X_0' + \theta Y_0'  \text{ mod.  } \z\\
[X,Y_0'] = Y_0' - \theta X_0' \text{ mod.  } \z.
\end{cases}
\end{align}
Remark that changing the representative of $\bar{X}$ does not affect these relations.  We may assume that the representatives $X_0$ and $Y_0$ were initially chosen such that (\ref{eq:similitude}) is valid.  We get
\begin{align*}
\Ad(e^{tX})X_0 = e^t(\cos(\theta t)X_0 + \sin(\theta t) Y_0) \text{ mod.  } \z
\end{align*}
and similarly for $Y_0$.  Let us note $X_t = \cos(\theta t)X_0 + \sin(\theta t) Y_0$. Since the projections of $X_0$ and $Y_0$ in $\n / \zn$ are linearly independent,  $P_x = \Span((X_0)_x,(Y_0)_x)$ is a plane on which the metric is everywhere positive-definite (see Remark \ref{rem:X0Positive}).
Hence,  we have positive constants $\alpha,\beta$ such that $\alpha \leq g_x(X_t,X_t) \leq \beta$ for all $t \in \R$ and $x \in M$. Recall that $Z$ is everywhere orthogonal to any vector field of $\n$. Therefore,
\begin{align*}
 \lambda(\phi_X^t,x) g_x(X_0,X_0) = e^{-2t} g_{\phi_X^t(x)}(X_t,X_t).  
 \end{align*}
Thus,  we get $\alpha/\beta \leq e^{2t} \lambda(\phi_X^t,x) \leq \beta / \alpha$ for all $x \in M$ and $t \in \R$.  We also have the relations $\lambda(\phi_X^t,x)g_x(X,X) = g_{\phi_X^t(x)}(X,X)$ for all $t$. So, we must have $g_x(X,X)=0$, for any $x \in M$.  

Since $[\gr,\z] \subset \z$, we have $\mu$ such that $[X,Z]=\mu Z$, therefore $(\phi_Z^t)_* X_x = X_{\phi_Z^t(x)} + \mu t Z_{\phi_Z^t(x)}$.  Using that $\phi_Z^t$ is isometric with respect to $g$ and $g(Z,Z)=0$ everywhere,  this implies $g_x(X,Z) = g_{\phi_Z^t(x)}(X,Z)$ and then
\begin{align*}
g_x(X,X) = g_{\phi_Z^t(x)}(X,X) + 2\mu t g_{x}(X,Z),
\end{align*}
so that either $\mu = 0$, or $g(X,Z) = 0$ everywhere.  If $\mu = 0$,  then we obtain $\lambda(\phi_X^t,x) g_x(X,Z) = g_{\phi_X^t(x)}(X,Z)$, which in turn implies $g_x(X,Z) = 0$ because $\lambda(\phi_X^t,x) \to +\infty$ as $t \to -\infty$.  In both cases,  $g(X,Z) = 0$ everywhere.

Now, we recall that the choice of $X$ as a representative of $\bar{X} \in \gr/\n$ was arbitrary.  In particular, the same conclusions follow for $X+Z$ and we obtain that $X_x$ is everywhere collinear to $Z_x$.  By Lemma \ref{lem:collinearVectorFields}, we obtain that $X$ is proportional to $Z$. This is the desired contradiction because $\ad(Z)$ acts trivially on $\n/\zn$.

Assume now that $\theta  = 0$.  It implies that there exists $X_1 \in \zn$ such that $[X,X_0] = X_0 + X_1$ and $X_0 \in \n \setminus \zn$.  Since $[X,X_1] \in \z$,  we get $[X,X_0'] = X_0' \text{ mod. } \z$,  where $X_0' = X_0+X_1$.  The same approach as above shows that $\lambda(e^{tX},x)$ is unbounded,  so $g(X,X) = 0$ identically and we derive the same contradiction.
\end{proof}

As the algebra $\gr / \n$ is abelian,  there exists a non-zero vector in $(\n/\zn)^{\C}$ which is common eigenvector to all the $\rho(\bar{X})^{\C}$, $X \in \gr$.  By the previous lemma, we obtain a linear function $\theta : \gr \rightarrow \R$ and non-zero vectors $\bar{X}_0,\bar{Y}_0 \in \gr / \zn$ such that 
\begin{align*}
\begin{cases}
\rho(\bar{X}) \bar{X}_0 &= \theta(X) \bar{Y}_0 \\
\rho(\bar{X}) \bar{Y}_0 &= -\theta(X) \bar{X}_0.
\end{cases}
\end{align*} 

\begin{remark}
\label{rem:X0Positive}
For any representatives $X_0,Y_0 \in \n$,  and for all $x \in M$,  the metric $g_x$ is positive definite in restriction to the plane spanned by $X_0,Y_0$ by Lemma \ref{lem:orthogonality_relations_loc_free}.
\end{remark}

Assume first that $\theta \neq 0$ and let us fix $X \in \gr$ such that $\theta(X) = 1$.  As in the proof of Lemma \ref{lem:purely_imaginary},  we can choose representatives $X_0,Y_0 \in \n$ such that
\begin{align*}
\begin{cases}
[X,X_0] = Y_0 \text{ mod. } \z \\
[X,Y_0] = -X_0 \text{ mod. } \z.
\end{cases}
\end{align*}
Thus,  $\Ad(e^{tX})$ acts as a rotation $R^t$ on $\Span(X_0,Y_0)$ modulo $\z$.  Let $\varphi :  M \rightarrow \R_{>0}$ be defined  by 
\begin{align*}
\varphi(x) = \int_S g_x(U,U) \d U
\end{align*}
where $S \subset \Span(X_0,Y_0)$ is the unit circle of the Euclidean norm for which $(X_0,Y_0)$ form an orthonormal basis and $\d U$ the Lebesgue measure on $S$.  For all $t \in \R$, $x \in M$,  and $U \in \Span(X_0,Y_0)$, we have
\begin{align*}
[(\phi_X^t)^*g]_x(U,U) & = g_{\phi_X^t(x)}(R^t(U),R^t(U))) 
\end{align*}
since $Z$ is everywhere isotropic and orthogonal to every vector field of $\n$. So,  we obtain
\begin{align*}
\lambda(\phi_X^t,x) \varphi(x) = \varphi(\phi_X^t(x)).
\end{align*}
By Lemma \ref{lem:[g,n]cap zn},  any $X' \in \ker \theta$, $\Ad(e^{tX'})$ acts trivially on $\Span(X_0,Y_0)$ modulo $\z$.   As a consequence,  $[(\phi_{X'}^t)^*g]_x(U,U) = g_{\phi_{X'}^t(x)}(U,U)$ for every $U \in \Span(X_0,Y_0)$, implying $\lambda(\phi_{X'}^t,x) \varphi(x) = \varphi(\phi_{X'}^t(x))$. Therefore, for any $X \in \gr$, the conformal distortion of $\phi_X^t$ with respect to $g$ is given by $\varphi(\phi_X^t(x)) / \varphi(x)$,  proving that $g / \varphi$ is preserved by all of $R$.  

Finally,  if $\theta = 0$ the arguments of the last paragraph show that if $X_0 \in \n \setminus \zn$ projects to an element in the kernel of all $\rho(\bar{X})$,  $X \in \gr$,  then $\psi(x) := g_x(X_0,X_0)$ is everywhere positive and $g/\psi$ is also $R$-invariant.

\subsection{Remark on the case of a non-compact Levi factor}
\label{ss:remarkLeviFactor}

We conclude this section with the proof of Proposition \ref{prop:essential_noncompact_levi}.  Assume that $N$ is inessential.  Let $\s$ be a Levi factor of $\g$ and assume that its non-compact part $\s'$ is non-trivial.  Consider the restriction of the adjoint representation of $\s'$ to $\n$.  Let $\a \subset \s'$ be an $\R$-split Cartan subalgebra.  If the representation is non-trivial,  considering a non-trivial weight,  there exists $X \in \a$ and $Y \in \n$ such that $[X,Y]=Y$.  Hence,  $g(X,X)>0$ everywhere by Lemma \ref{lem:affine_inessential}.  Applying Lemma 2.3 of \cite{pecastaing_smooth} to couples $(X,Y_{\alpha})$ where $Y_{\alpha}$ is any element in a weight space of $\ad(\a)$,  we obtain a basis of $\g$ whose elements preserve the metric $g/g(X,X)$.  So $G$ is inessential.

Therefore,  when $G$ is essential,  $\s'$ centralizes $\n$.  If $N$ was non-abelian,  then any $X \in \n \setminus \zn$ would satisfy $g(X,X) > 0$ and because $\s'$ centralizes $X$,  it would be inessential,  implying that $G$ is inessential by \cite{pecastaing_smooth} (Proof of Corollary 1.2).

\section{Essential actions}
\label{s:essential}

From now on,  we will consider compact Lorentzian manifolds admitting essential actions of Lie groups.  We still note $(M,g)$ a compact Lorentzian manifold of dimension $n \geq 3$ and we assume that a solvable Lie group $R$ acts conformally and essentially on $M$.  The aim of this section is to provide certain sufficient conditions for the existence of elements or one-parameter subgroups of $R$ admitting a singularity with a non-linear unipotent holonomy (mainly,  the element or the corresponding flow will have trivial $1$-jet at this point).  We will then conclude that there exists a non-empty conformally flat open subset via anterior results.

\subsection{Preliminaries: associated Cartan geometry modeled on $\Ein^{1,n-1}$}
\label{ss:cartanGeometryConformal}

Recall that the (projective model of the)  \textit{Lorentzian Einstein Universe} is the smooth projective quadric $\{Q=0\} \subset \R P^{n+1}$ where $Q$ is a quadratic form of signature $(2,n)$ on $\R^{n+2}$.  It inherits a natural conformal Lorentzian structure such that $\Conf(\Ein^{1,n-1}) \simeq \PO(2,n)$.  From now on,  we will denote $\PO(2,n)$ by $G$.

\subsubsection{Equivalence principle}

We introduce now a central tool in conformal geometry.  Let $P<G$ denote the stabilizer of a null-line $x_0$ in $\R^{2,n}$,  let $\p$ denote its Lie algebra.  We denote by $P^+ \simeq \R^n$ the nilradical of $P$ and by $G_0 \simeq \R_{>0} \times O(1,n-1)$ a section of $G/P^+$,  so that $P \simeq G_0 \ltimes P^+$.

The following theorem says that a general Lorentzian conformal structure $(M,[g])$ can be interpreted as a \textit{curved version} of $\Ein^{1,n-1}$.

\begin{theorem*}[Equivalence principle for conformal Lorentzian structures]
Let $(M^n,[g])$ be a conformal Lorentzian structure in dimension $n \geq 3$.   Then,  there exists a $P$-principal fibration $\pi : \hat{M} \to M$ and a $1$-form $\omega \in \Omega^1(\hat{M},\g)$ verifying
\begin{itemize}
\item $\forall \hx \in \hat{M}$,  $\omega_{\hx} : T_{\hx} \hat{M} \to \g$ is a linear isomorphism ;
\item  $\forall A \in \p$,  $\omega(A^*) \equiv A$,  where $A^*$ stands for the  fundamental field associated to $A$ ;
\item $\forall \hx \in \hat{M}$,  $\forall p \in P$, \ $(R_p)^* \omega = \Ad(p^{-1}) \omega$,  with $R_p$ standing for the right $P$-action on $\hat{M}$.
\end{itemize}
Additionally,  a technical normalization condition is required,  making this correspondence one-to-one.  We then have the following lifting property: 

Let $f \in \Diff(M)$ be a diffeomorphism.  Then,  $f$ is conformal with respect to $[g]$ if and  only if there exists a bundle automorphism $\hat{f}$ such that $\pi \circ \hat{f} = f \circ \pi$ and $(\hat{f})^* \omega = \omega$.   In this situation, $\hat{f}$ is uniquely determined,  and called the lift of $f$.

Let $X$ be a vector field on $M$.  Then,  $X$ is conformal if and only if there exists a vector field $\hat{X}$ defined on $\hat{M}$ such that $(R_p)^* \hat{X} = \hat{X}$ for all $p \in P$, $\pi_* \hat{X} = X$,  and $\mathcal{L}_{\hat{X}} \omega = 0$.  In this situation,  $\hat{X}$ is uniquely determined,  and called the lift of $X$.
\end{theorem*}

\begin{remark}
When there is no possible confusion,  we will however use the same notation for the lift of $X$.  For instance,  	we will frequently denote by $\omega_{\hx}(X)$ the evaluation at a point $\hx \in \hat{M}$ of the Cartan connection on the lift of a conformal vector field $X$.
\end{remark}

The curvature of the associated Cartan geometry is the \textit{horizontal} $2$-form $\Omega \in \Omega^2(\hat{M},\g)$ defined by $\Omega = \d \omega + \frac{1}{2}[\omega,\omega]$.  It is classically known (\cite{sharpe} for instance) that the curvature identically vanishes if and only if $(M,[g])$ is locally conformally equivalent to $\Ein^{1,n-1}$.

\begin{lemma}[\cite{bader_frances_melnick} Lem. 2.1]
\label{lem:iota_crochet}
Let $X,Y$ be two conformal vector fields.  Then,  for all $\hx \in \hat{M}$,  we have
\begin{align*}
\omega_{\hx}([\hat{X},\hat{Y}]) = - [\omega_{\hx}(\hat{X}),\omega_{\hx}(\hat{Y})] + \Omega_{\hx}(\hat{X},\hat{Y}).  
\end{align*}
In particular,   if $\hat{X}_{\hx}$ or $\hat{Y}_{\hx}$ is vertical,  then $\omega_{\hx}([\hat{X},\hat{Y}]) = - [\omega_{\hx}(\hat{X}),\omega_{\hx}(\hat{Y})]$.
\end{lemma}

Note that we recover here that a Lie algebra of conformal vector fields of a conformally flat Lorentzian manifold can be embedded into $\g=\so(2,n)$.  This identity even implies that if the Cartan curvature vanishes at a single point,  then we obtain an embedding.

\subsubsection{Restricted root-space decomposition of $\so(2,n)$}
\label{sss:roots}

Recall that the restricted root-system of $\g = \so(2,n)$ is $B_2$,  \textit{i.e.} isomorphic to $\{\pm \alpha, \pm \beta,  \pm \alpha \pm \beta\}$ with $\alpha \perp \beta$ and $|\alpha|=|\beta|$.   We introduce linear coordinates on $\R^{2,n}$ which we fix once and for all.  Let $(e_0,\ldots,e_{n+1})$ be a basis of $\R^{2,n}$ in which the quadratic form $Q$ reads $ 2u_0u_{n+1}+2u_1u_n+u_2^2+\cdots+u_{n-1}^2$ and such that $P < G$ is the stabilizer of $x_0:=[e_0]$.  
We denote by $\a$ the $\R$-split Cartan subalgebra  
\begin{align*}
\a = 
\left \{
A=
\begin{pmatrix}
\lambda & & & & \\
 & \mu & & & \\
 & & \mathbf{0} & & \\
 & & & -\mu & \\
 & & & & -\lambda
\end{pmatrix}
, \ \lambda,\mu \in \R
\right \}.
\end{align*}
We denote by $\alpha$ and $\beta$ the restricted-roots $\alpha(A)=\lambda$ and $\beta(A)=\mu$.  The compact part of the centralizer of $\a$ is identified as
\begin{align*}
\m =
\left \{
\begin{pmatrix}
0 & & & & \\
 & 0 & & & \\
 & & X & & \\
 & & & 0 & \\
 & & & & 0
\end{pmatrix}
, \ X \in \so(n-2)
\right \},
\end{align*}
and the restricted root-spaces of $\a$ are located as indicated below
\begin{align*}
\begin{pmatrix}
\a & \g_{\alpha - \beta} & \g_{\alpha} & \g_{\alpha+\beta} & 0 \\
\g_{\beta - \alpha}   & \a & \g_{\beta} & 0 & \g_{\alpha+\beta} \\
  \g_{-\alpha} & \g_{- \beta} & \m & \g_{\beta} & \g_{\alpha }  \\
 \g_{- \alpha - \beta}  & 0 & \g_{- \beta} & \a & \g_{\alpha - \beta} \\
  0  & \g_{- \alpha - \beta} & \g_{- \alpha} & \g_{\beta - \alpha} & \a
\end{pmatrix}.
\end{align*}
The Lie algebra of $P$ is $\p = \a \oplus \m \oplus \g_{-\beta}\oplus \g_{\beta} \oplus \g_{\alpha-\beta} \oplus \g_{\alpha} \oplus \g_{\alpha+\beta}$ and the nilradical $P^+$ of $P$ has Lie algebra $\p^+ = \g_{\alpha-\beta} \oplus \g_{\alpha} \oplus \g_{\alpha+\beta}$.

The element $A_{\alpha} \in \a$ such that $\alpha(A_{\alpha}) = 1$ and $\beta(A_{\alpha}) = 0$ defines a grading $\g = \g_{-1} \oplus \g_0 \oplus \g_1$, where $\g_k = \{X \in \g  \ | \ [A_{\alpha},X]=kX\}$,  for $k =-1,0,1$.   We have $\p = \g_0 \ltimes \g_1$,  $\g_1=\p^+$,  $\g_{-1} = \g_{-\alpha+\beta} \oplus \g_{-\alpha} \oplus \g_{-\alpha-\beta}$ and $\g_0 \simeq \R \oplus \so(1,n-1)$,  where the $\R$-factor is $\R.A_{\alpha}$.

\subsubsection{Stereographic projections and holonomy of conformal maps and vector fields with a singularity}

\begin{definition}
Let $X$ be a conformal vector field of $(M,[g])$,  let $\hat{X}$ be its lift to the Cartan bundle.  Assume that $X$ has a singularity $x \in M$.  For any $\hx \in \pi^{-1}(x)$,  the \textit{holonomy} of $X$ at $\hx$ is defined as the element $\omega_{\hx}(X) \in \p$.

If $f \in \Conf(M,[g])$ fixes a point $x \in M$,  and if $\hx \in \pi^{-1}(x)$,  then the holonomy of $f$ at $\hx$,  denoted $\hol^{\hx}(f)$,  is the unique element $p \in P$ such that $\hat{f}(\hx) = \hx.p^{-1}$.  
\end{definition}

The correspondence $f \in \Conf(M,[g])_x \mapsto \hol^{\hx}(f) \in P$ is then an injective Lie group homomorphism. Of course,  when $X(x) = 0$,  the holonomy of $\phi_X^t$ at $\hx$ is $e^{tX_h}$.

Since $\omega$ defines a global framing on $\hat{M}$,  $X$ is completely determined by the evaluation of $\hat{X}$ at any point $\hx \in \hat{M}$.  Thus,  when it admits a singularity,  $X$ is determined by the  data of its holonomy.  The ideal situation is when a neighborhood of $x$ is conformally flat: $X$ is then locally conjugate to the conformal vector field $X_h$ of $G/P$ near its singularity $eP$.  However,  relating directly the dynamics of $X$ to its holonomy is a difficult task in general.  

The Cartan connection induces a linear identification $\varphi_{\hx} : T_xM \to \g/\p$,  such that for any $f \in \Conf(M,[g])$ with a fixed point $x$ and if $p  = \hol^{\hx}(f)$ for $\hx \in \pi^{-1}(x)$, we have
\begin{align*}
\varphi_{\hx} \circ \d_x f = \bar{\Ad}(p) \circ \varphi_{\hx}
\end{align*}
where $\bar{\Ad}$ is the representation of $P$ on $\g/\p$ induced by the adjoint representation.  So,  $\d_x f$ is conjugate to $\Ad(g_0)|_{\g{-1}}$,  where $g_0$ is the $G_0$ component of $\hol^{\hx}(f)$.  In particular,  $\d_x f = \id$ if and only if $\hol^{\hx}(f) \in P^+$,  and similarly for conformal vector fields.

We describe the dynamics of the holonomy of $\hol^{\hx}(f)$ or $X^h$ via \textit{stereographic projections}.   Let $(e_0,\ldots,e_{n+1})$ be a basis of $\R^{2,n}$ in which the quadratic form $Q$ reads $ 2u_0u_{n+1}+2u_1u_n+u_2^2+\cdots+u_{n-1}^2$ and such that $P < G$ is the stabilizer of $x_0:=[e_0]$.  We denote by  $s_{x_0} : \R^{1,n-1} \to M_{x_0}$ the stereographic projection given by
\begin{align*}
s_{x_0}(v_1,\ldots,v_n) = \left [ -\frac{\langle v,v \rangle}{2} : v_1 : \cdots : v_n : 1 \right ],
\end{align*}
where $\langle v,v \rangle = 2v_1v_n + v_2^2 + \cdots + v_{n-1}^2$. It is a conformal diffeomorphism between $\R^{1,n-1}$ and an open-dense subset subset of $\Ein^{1,n-1}$ called the Minkowski patch $M_{x_0}$ associated to $x_0$,  whose complement is the light-cone based at that point (see for instance \cite{frances_melnick_normal_forms} §2.1.2 or \cite{pecastaing_conformal_projective} §2.1).  The parabolic subgroup $P \simeq \CO(1,n-1) \ltimes \R^n$ preserves $M_{x_0}$ and $s_{x_0}$ conjugates its action to the affine action of $\CO(1,n-1) \ltimes \R^n$ on $\R^{1,n-1}$.  

By Liouville Theorem,  every conformal diffeomorphism from $\R^{1,n-1}$ to an open subset of $\Ein^{1,n-1}$ is of the form $g.s_{x_0}$,  for some $g \in G$.  Such maps are called stereographic projections,  and the pole of projection of $g.s_{x_0}$ is $g.x_0$.  For $x \in \Ein^{1,n-1}$,  if $G_x$ denotes the stabilizer of $x$,  then the nilradical $\g_x^+$ of $\g_x$ is characterized as the subspace of conformal vector fields $X$ of $\Ein^{1,n-1}$ fixing $x$ and such that $s^*X$ is translation of $\R^{1,n-1}$ for some (equivalently any) stereographic projection $s$ with pole $x$. 

\begin{definition}
For any $x \in \Ein^{1,n-1}$,  an element $X \in \g$ is called a \textit{translation} of the Minkowski patch $M_x$ if $X \in \g_x^+$,  or equivalently $s^*X$ is a translation of $\R^{1,n-1}$ for any stereographic projection $s$ with pole $x$.  It is said to be a space-like/time-like/light-like translation if $s^*X$ is a translation of $\R^{1,n-1}$ with the same property.
\end{definition}

\subsubsection{Local linearization,  non-linear unipotent point-wise holonomy}

We identify $P$ as the subgroup $\Ad(G_0)|_{\p^+} \ltimes \p^+$ of the affine group of $\p^+$,  and its Lie algebra correspondingly.  A one-parameter subgroup $\{e^{tX}\} \subset P$ is conjugate to a one-parameter subgroup of $G_0$ if and only if its affine action on $\p^+$ has a fixed point.

\begin{proposition}[\cite{frances-localdynamics},  Prop. 4.2]
\label{prop:linearization}
Let $f \in \Conf(M,[g])$ fixing $x$ and let $\hx \in \pi^{-1}(x)$.  Then,  $f$ is locally linearizable near $x$ if and only if $\hol^{\hx}(f)$ is conjugate to an element of $G_0$.  
\end{proposition}

\begin{remark}
\label{rem:holonomy_linear}
We have the analogue for the linearization of a conformal vector field $X$.  Note that writing $X^h = (X_0, X_1) \in \g_0 \ltimes \p^+$,  $X^h$ is conjugate to an element of $\g_0$ if there exists $Y_1 \in \p^+$ such that $X_1 = [X_0,Y_1]$,  and in this case $\Ad(e^{Y_1}) X^h = X^h - [X_0,Y_1] = X_0$.
\end{remark}

We will use the following results for essential singularities of conformal vector fields:

\begin{proposition}[\cite{frances_melnick_normal_forms}]
\label{prop:holonomy_uniptotent}
Let $X$ be a conformal vector field of a Lorentzian manifold admitting a singularity $x_0$.  Let $X_h \in \p$ be its holonomy at some given point in the fiber of $x_0$.  If $e^{tX_h}$ is a unipotent one-parameter subgroup of $P$,  and is not conjugate to a one-parameter subgroup of $G_0$,  then an open subset of $M$ containing $x_0$ in its closure is conformally flat.
\end{proposition}

\begin{proof}
This is proved in \cite{frances_melnick_normal_forms},  §5.3.   Although the authors give a statement (Theorem 1.2) valid only in analytic regularity,  they use this analyticity assumption to deduce conformal flatness of the whole manifold from that of an open subset,  and importantly to reduce their problem to the case of a vector field with unipotent holonomy.  The rest of their proof is then valid for smooth metrics,  yielding a proof of this proposition.
\end{proof}

A special situation where this proposition applies is when the $\g_0$-part of $X_h$ is trivial,  \textit{i.e.} if the flow of $X$ has trivial $1$-jet at $x$.  This case is in fact covered by Theorem 4.1 of \cite{frances_melnick_normal_forms} which treats every metric signature.  Moreover,  the proof of this result can be adapted to the following setting (a stronger result is announced by the authors at the moment,  see for instance Theorem 9.1 of \cite{frances_melnick_lichne-anal3d}  for the $3$-dimensional case). 

\begin{proposition}
\label{prop:discrete_holonomy}
Let $(M,g)$ be a Lorentzian manifold and $f \in \Conf(M,[g])$ be such that $f(x)=x$ and $\hol^{\hx}(f) \in \exp(\p^+)$ for some point $x$ and some (equivalently all) $\hx$ in the fiber over $x$.  Then,  an open subset containing $x$ in its closure is conformally flat.
\end{proposition}

\subsection{$\heis$-triples in $\so(2,n)$}

We prove here an algebraic property which will be used later.  Recall that we note $P < \PO(2,n)$ the stabilizer of the isotropic line $x_0 = [1:0:\cdots:0]$ in the fixed coordinates of $\R^{2,n}$ that we chose in Section \ref{sss:roots}.

\begin{proposition}
\label{prop:heis_triple_so2n}
Let $(X,Y,Z)$ be an $\heis$-triple in $\so(2,n)$ such that $Z \in \p$. Then,  $Z$ is a light-like translation of the Minkowski patch $M_{x_0}$.
\end{proposition}

Otherwise stated,  there exists $p \in P$ such that $\Ad(p)Z \in \g_{\alpha+\beta}$.

\begin{proof}
We use the general Proposition 3.3 of \cite{frances_melnick_nilpotent} which guarantees the conclusion here,  modulo conjugacy in $G$.  It is enough to show that the conjugacy can be realized by an element of $P$,  since the cone of null-translations of $P$ is $\Ad(P)$-invariant.

Let $g \in \PO(2,n)$ such that $\Ad(g)X \in \g_{\alpha+\beta}$.  Recall that we identify the Lie algebra $\so(2,n)$ with that of conformal vector fields of $\Ein^{1,n-1}$.

\begin{lemma}
Let $X \in \so(2,n)$.  If $X$ is a light-like translation of some Minkowski patch $M_x$,  then its fixed points form a light-like geodesic $\Delta \subset \Ein^{1,n-1}$,  \textit{i.e.} the projectivization of a null plane in $\R^{2,n}$.

For any $y \in \Delta$,  $X$ is still a light-like translation of the Minkowski patch $M_y$.
\end{lemma}

\begin{proof}
For the first point,  we just have to see that it is true for $X = X_{\alpha+\beta} \in \p^+$ by homogeneity.  It is then a straightforward verification.

For the second point,  we may assume that $x = x_0 = [1:0:\cdots:0]$ and $y=[0:1:0:\cdots:0]$ because $x$ and $y$ belong to a common light-like geodesic.  A stereographic projection at $y$ is then given by a similar formula as for the projection at $x_0$,  and an explicit computation gives the result.
\end{proof}

Let $x _0 \in \Ein^{1,n-1}$  be the point such that $P=G_{x_0}$ and let $x=g.x_0$. Let $s_0=s_{x_0}$ be the stereographic projection introduced above and let $s = g s_0$.  Then,  $(s_0)^* Z = s^*((g^{-1})^*Z) = s^*(\Ad(g)Z)$.  Let $Z' = \Ad(g)Z$.  Then $Z'$ fixes $x$ because $Z$ fixes $x_0$,  and since $Z'$ also fixes $x_0$ and is a light-like translation of $M_{x_0}$,  the previous lemma implies that $Z'$ is also a light-like translation of $M_x$.  So,  $(s_0)^*Z$ is a light-like translation, as claimed.
\end{proof}

\subsection{Light-like singularities in $\n_k$} 
\label{ss:singularitiesInNk}

We prove now point (3) in Theorem \ref{thm:nilpotent}.
Let $N$ be a non-abelian nilpotent Lie group of nilpotence degree $k+1$.  Assume that $N$ acts locally faithfully and conformally essentially on a compact Lorentzian manifold $(M,g)$.  

\begin{lemma}
There exist $X \in \n_{k-1}$ and $Y \in \n$ such that the vector field $[X,Y]$ is non-zero and has a singularity.
\end{lemma}

\begin{proof}
If not,  then we can apply Lemma \ref{lem:heis3_triple_inessential} to any $\heis(3)$-triple of the form $(X,Y,[X,Y])$ with $X \in \n_{k-1}$ and $Y \in \n$ non-commuting and obtain that $[X,Y]$ is everywhere isotropic.  We can then apply Lemma \ref{lem:nk_dim1} to conclude that $\dim \n_k=1$.  But this contradicts the essentiality of $N$ by Proposition \ref{prop:inessential_nonabelian} since $\n_k = \R.[X,Y]$ would act locally freely.
\end{proof}

Let $\h$ be the Lie algebra of conformal vector fields of $(M,g)$ spanned by $(X,Y,Z)$,  with $Z := [X,Y]$ given by the previous lemma.  Let $H<\Conf(M,g)$ denote the associated integral subgroup.  Note that the compact subset $K=\{x \in M \ : \ Z_x=0\}$ is $H$-invariant.

\begin{lemma}
\label{lem:Znull_translation}
There exists $x \in K$ at which the holonomy of $Z$ is a null-translation in $\p^+$.  
\end{lemma}

\begin{proof}
We see $H$ as an integral subgroup of $\Conf(M,[g])_0$ and we note $H'$  the closure of $H$ in $\Conf(M,[g])_0$.  As seen in Lemma \ref{lem:closure_nilpotent},  $H'$ is a $2$-step nilpotent Lie subgroup of $\Conf(M,[g])$.  Let $X \in \h$ be a non-central element,  $Y \in \h$ such that $[X,Y]=Z$ and let $S = \{e^{tX}\}_{t\in \R}$.  We apply Theorem \ref{thm:bfm} to the pair $(S,H')$.   Let $\mu$ be a finite $S$-invariant measure supported in $K$. As $H'$ is nilpotent,  $\bar{S}^d = \Ad_{\h'}(S)$,   and then for $\mu$-almost every $x \in M$ and for all $\hx \in \pi^{-1}(x)$,  we obtain a one parameter subgroup $\{e^{tX'}\} \subset P$ such that the map $\iota_{\hx} : \h' \to \g$ conjugates the action of $\Ad(e^{tX})$ on $\h'$ and that of $\Ad(e^{tX'})$ on $\iota_{\hx}(\h')$.

In particular,  we obtain that $[X',\iota_{\hx}(Y)]=\iota_{\hx}(Z)$,  $[X',\iota_{\hx}(Z)]=0$.  Because $x \in K$,  the lift of $Z$ to $\hat{M}$ is tangent to the fiber of $\hat{M}$,  meaning that $\iota_{\hx}(Z) \in \p$.  Applying Lemma \ref{lem:iota_crochet},  we get $[\iota_{\hx}(Y),\iota_{\hx}(Z)]=0$.  Hence,  $(X',\iota_{\hx}(Y),\iota_{\hx}(Z))$ generates a copy of $\heis(3)$ in $\so(2,n)$ whose center is contained in $\p$.  By Proposition \ref{prop:heis_triple_so2n}, we obtain that $\iota_{\hx}(Z)$ is a null-translation in $\p^+$ as announced.
\end{proof}

Therefore,  by Proposition \ref{prop:holonomy_uniptotent},  an open subset of $M$ is conformally flat.  In particular,  we have an embedding $\n \hookrightarrow \so(2,n)$.  Applying Proposition 3.3 of \cite{frances_melnick_nilpotent},  it follows that $k \leq 3$ and $\n_k$ is a line,  and this concludes the proof of Theorem \ref{thm:nilpotent}.

\subsection{Semi-simplicity of $\rho$ and reduction to a semi-direct product.}
\label{ss:semi-simplicity}

Let $R<\Conf(M,g)$ be a connected solvable Lie subgroup. We assume that $R$ is not nilpotent and has abelian nilradical $N$.  By Propositions \ref{prop:inessential_abelian} and \ref{prop:locallyFreeOpenDense},  $N$ has dimension at most $n$ and there exists $x\in M$ with $\dim N.x \leq 1$, and if $\dim N.x = 1$ then it is light-like. 

We note $\rho : X \in \gr \mapsto \ad(X)|_{\n} \in \gl(\n)$.  By assumption,  $\rho$ is trivial in restriction to $\n$ and has abelian image since $[\gr,\gr] \subset \n$.  

\begin{proposition}
\label{prop:rho_semisimple}
Let $X \in \gr$. If $\rho(X)$ is not a semi-simple element of $\gl(\n)$, then a vector field of $\n$ admits a singularity at which its holonomy is a light-like translation.  In particular,  some non-empty open subset is confomally flat.
\end{proposition}

\begin{proof}
Let $\h = \R.X \oplus \n$.  We first observe that the associated connected subgroup $H < R$ is closed.  To see it,  let $H' = \{e^{tX}e^Y, \ t \in \R,  \ Y \in \n\}$.  It is a connected subgroup,  contained in $H$.  If we prove that $H'$ is closed,  we will have $H'=H$.  Let then $g = \lim e^{t_n X}e^{Y_n}$ be a limit of elements of $H'$.  In particular,  we have a bounded sequence $g_n \in N_R(\n)$ such that $e^{t_n X} = g_n e^{-Y_n}$.  Since $\n$ is abelian,  we get that $\Ad(e^{t_n X}) |_{\n} = \Ad(g_n)|_{\n}$ must be bounded too.  Since,  the Jordan decomposition of $\rho(X)$ has a non-zero nilpotent component and $\Ad(e^{t_n X})|_{\n} = e^{t_n \rho(X)}$,  $(t_n)$ must be bounded.  So,  up to an extraction $e^{Y_n}$ converges,  and the limit is in $N = \{e^Y, \ Y \in \n\}$, the latter being closed because $\n$ is the  (abelian) nilradical of $\gr$

Let $\ad(X)|_{\h} = X_{ss} + X_u$ be the Jordan decomposition in $\gl(\h)$.  By hypothesis,  $X_u \neq 0$.  We apply now Theorem \ref{thm:bfm} to the pair $(H,S)$ where $S = \{e^{tX}\}$.  Since the Zariski closure of $\Ad_{\h}(S)$ contains $\{e^{tX_u}\}$ which is Zariski closed,  then so does the discompact radical $\bar{S}^d$ (\cite{shalom},  Prop. 1.4).   
Since $K: = \{x \in M \ |  \dim(N.x) \leq 1\}$ is closed and $S$-invariant,  considering an $S$-invariant measure supported in $K$,  we deduce that there exists $x \in K$ such that for all $\hx \in \pi^{-1}(x)$,  there exists $X' \in \p$ such that $[X',\iota_{\hx}(Y)] = \iota_{\hx}(X_u(Y))$ for all $Y \in \h$.  Choose $Y \in \n$ such that $Z:=X_u(Y) \neq 0$ and $X_u^2(Y)=0$.  If we note $Y'=\omega_{\hx}(Y)$ and $Z' = \omega_{\hx}(Z)$,  we then have $[X',Y']=Z'$,  $[X',Z']=0$.  Since $\dim N.x \leq 1$,  $\Span(Y',Z') \cap \p \neq 0$,  so we must have $Z' \in \p$ because if $\alpha Y' + \beta Z' \in \p$,  then $[X',\alpha Y' + \beta Z'] = \alpha Z' \in \p$.  By Lemma \ref{lem:iota_crochet},  we get $[Y',Z']= \iota_{\hx}([Y,Z])=0$,  so $(X',Y',Z')$ is an $\heis$-triple of $\so(2,n)$ with $Z' \in \p$.  By Proposition \ref{prop:heis_triple_so2n},  we get that the holonomy of $Z$ at $x$ is a light-like translation.
\end{proof}

\begin{proposition}
\label{prop:split_exact_sequence}
Assume that every element of $\rho(\gr)$ is semi-simple.  Then,   the short exact sequence $\n \to \gr \to \gr / \n $ is split modulo the center, \textit{i.e.} there exists $\a \subset \gr$ such that $\gr  = \a \oplus \n$ and $[\a,\a] \subset \z(\gr)$. In particular,  if $[\a,\a] \neq 0$, then $\gr$ contains an $\heis$-triple.  Moreover,  every $\heis$-triple of $\gr$ is essential.  
\end{proposition}

\begin{proof}
By our hypothesis on semi-simplicity,  and the fact that elements in $\gr$ commute pairwise,    we have a common (complex) diagonalisation basis of $\n$ of the form  
\begin{align*}
(X_1,Y_1,\ldots,X_r,Y_r,T_1,\ldots,T_s,Z_1,\ldots,Z_t)
\end{align*}
and linear functions $\lambda_i,\mu_i$ and $\nu_j$ in $\gr^*$ such that $\mu_i \neq 0$ and $\nu_j \neq 0$ and for all $X \in \gr$
\begin{align*}
\begin{cases}
[X,X_i] & = \lambda_i(X) X_i + \mu_i(X)Y_i \\
[X,Y_i] & = -\mu_i(X)X_i + \lambda_i(X) Y_i \\
[X,T_j] & = \nu_j(X)T_j \\
[X,Z_k] & = 0
\end{cases}
\end{align*}
Let $d = \dim \gr/\n$ and let us choose $X^1,\ldots,X^d \notin \bigcup_i \ker \mu_i \cup \bigcup_j \ker \nu_j$ which project to a basis of $\gr / \n$. For all $l,m$, we have
\begin{align*}
[X^l,X^m] = \sum_i (a_{lm}^iX_i + b_{lm}^iY_i )+ \sum_j c_{lm}^j T_j + \sum_k d_{lm}^kZ_k.
\end{align*}
Now, for $\tilde{X}^l=X^l + \sum_i (\alpha_l^i X_i + \beta_l^iY_i) + \sum_j \gamma_l^jT_j$ we obtain 
\begin{align*}
[\tilde{X}^l,\tilde{X}^m] & =   \sum_i(a_{lm}^i + \lambda_i(X^l)\alpha_m^i - \lambda_i(X^m)\alpha_l^i -\mu_i(X^l)\beta_m^i + \mu_i(X^m)\beta_l^i)X_i \\
& + \sum_i (b_{lm}^i + \lambda_i(X^l)\beta_m^i - \lambda_i(X^m) \beta_l^i + \mu_i(X^l) \alpha_m^i - \mu_i(X^m)\alpha_l^i)Y_i \\
& + \sum_j (c_{lm}^j + \nu_j(X^l)\gamma_m^j - \nu_j(X^m) \gamma_l^j) T_j + \sum_k d_{lm}^k Z_k.
\end{align*}
Let us see that we can adjust the coefficients $\alpha_l^i,\beta_l^i,\gamma_l^i$ such that $[\tilde{X}^l,\tilde{X}^m]\in \Span(Z_1,\ldots,Z_t)$. For a fixed index $i$,  we impose for all $l,m$ the equations 
\begin{align*}
\begin{cases}
a_{lm}^i + \lambda_i(X^l)\alpha_m^i - \lambda_i(X^m)\alpha_l^i -\mu_i(X^l)\beta_m^i + \mu_i(X^m)\beta_l^i = 0 \\
b_{lm}^i + \lambda_i(X^l)\beta_m^i - \lambda_i(X^m) \beta_l^i + \mu_i(X^l) \alpha_m^i - \mu_i(X^m)\alpha_l^i = 0,
\end{cases}
\end{align*}
which we see as the real part and imaginary part of a same equation.  Let us remove temporarily the index $i$.  Our relations read for all $l,m$
\begin{align}
\label{eq:cohom}
A_{lm} = \zeta(X^m) Z_l - \zeta(X^l)Z_m,
\end{align}
where $\zeta = \lambda+i\mu$, $A_{mn} = a_{mn}+ib_{mn}$, and $Z_l  = \alpha_l + i \beta_l$. Note that $A_{mn} = - A_{nm}$. The Jacobi relation between $X^l,X^m,X^n$ then reads
\begin{align*}
A_{mn}\zeta(X^l)  + A_{nl}\zeta(X^m) + A_{lm}\zeta(X^n) = 0.
\end{align*}
Therefore, equation (\ref{eq:cohom}) for $l,m \geq 2$ follows from the equation (\ref{eq:cohom}) for $(1,l)$, $(1,m)$ and the Jacobi relation between $X^1,X^l,X^m$.  So, the solutions $(Z_1,\ldots,Z_d)$ of our systems are the (complex) multiples of 
\begin{align*}
(\zeta(X^1), \frac{A_{2,1}}{\zeta(X^1)} + \zeta(X^2), \ldots,\frac{A_{d,1}}{\zeta(X^1)} + \zeta(X^d)).
\end{align*}
Consequently,  if we choose 
\begin{align*}
\alpha_m^i & = \lambda_i(X^m) + \frac{\lambda_i(X^1)a_{m,1}^i + \mu_i(X^1)b_{m,1}^i}{\lambda_i(X^1)^2+\mu_i(X^1)^2} \\
\beta_m^i & = \mu_i(X^m) + \frac{\lambda_i(X^1)b_{m,1} - \mu_i(X^1)a_{m,1}^i}{\lambda_i(X^1)^2+\mu_i(X^1)^2} \\
\gamma_m^i & = \nu_i(X^m) + \frac{c_{m,1}^i}{\nu_i(X^1)},
\end{align*}
then, we get that $[\tilde{X}^l,\tilde{X}^m] \in \z(\gr)$ for all $l,m$.  

Thus, there exists a section $\a$ of $\gr \to \gr / \n$ such that $[\a,\a] \subset \z(\gr)$.  Assume that $\z(\gr) \neq 0$ and let $X,Y \in \a$ be such that $[X,Y] =:Z \neq 0$.  Then $(X,Y,Z)$ is an $\heis$-triple.

We prove now that any $\heis$-triple of $\gr$ is essential.  Let us assume by contradiction that some $\heis$-triple $(X,Y,Z)$ preserves a metric $g$ in the conformal class.  We know that this triple of vector fields acts everywhere locally freely, with degenerate orbits, and that the orbits of $Z$ give the kernel (Remark \ref{rem:isometric_heis}).  In particular, $g_x(U,U)>0$ everywhere, for every $U \in \Span(X,Y,Z) \setminus \R.Z$.  

For $1 \leq i \leq r$, for all $V \in \Span(X_i,Y_i)$,  we have $(\phi_U^t)_* V_x = e^{-\lambda_i(U)t} [R^{-\mu_i(U)t}V]_{\phi_U^t(x)}$ where $R^{\theta}V$ stands for the standard rotation in the plane $\Span(X_i,Y_i)$.  Thus,
\begin{align*}
g_x(Z,V) & = e^{-\lambda_i(U)t} g_{\phi_U^t(x)}(Z,R^{-\mu_i(U)t}V) \\
g_x(V,V) & = e^{-2\lambda_i(U)t} g_{\phi_U^t(x)}(R^{-\mu_i(U)t}V,R^{-\mu_i(U)t}V).
\end{align*}
Therefore, if $\lambda_i(U) \neq 0$,  then for all $V \in \Span(X_i,Y_i)$, $g(V,V) = g(Z,V)= 0$. By Lemma \ref{lem:collinearVectorFields},  it implies that $V$ is a multiple of $Z$, a contradiction.  So $\lambda_i(U) = 0$ for all $U \in \Span(X,Y)$.  Pick now $U_0 \in \Span(X,Y)$ such that $\mu_i(U_0) = 0$. Then $[U_0,V]=0$ for all $V \in \Span(X_i,Y_i)$.  So,  every $V \in \Span(X_i,Y_i)$ is  a Killing vector field of $g' := g/g(U_0,U_0)$.  Now for any $U$, we have $(\phi_V^t)_* U_x = U_{\phi_V^t(x)} + t[U,V]_{\phi_V^t(x)}$ 
 and using that $\phi_V^t$ is isometric with respect to $g'$, we deduce similarly as above that $[U,V]$ is everywhere light-like and orthogonal to $Z$.  It follows that $\mu_i(X)=\mu_i(Y) =0$ because if for instance $\mu_i(X) \neq 0$,  then Lemma \ref{lem:collinearVectorFields} would imply that $Y_i = \frac{1}{\mu_i(X)} [X,X_i]$ and $X_i = - \frac{1}{\mu_i(X)}[X,Y_i]$ are collinear to $Z$. Thus, for all $1 \leq i \leq r$,  $\lambda_i$ and $\mu_i$ vanish on $\Span(X,Y)$.

For $1 \leq j \leq s$,  from $(\phi_X^t)_* (T_j)_x = e^{-\nu_j(X)t}(T_j)_{\phi_X^t(x)}$,  we deduce that  $\nu_j(X) = 0$,  because if $\nu_j(X) \neq 0$,  then we could prove similarly as above that $T_j$ is  a multiple of $Z$,  a contradiction because $[X,T_j] = \nu_j(X)T_j \neq 0$.  Symmetrically,  $\nu_j(Y) = 0$ for all $j$.  Finally,  both $X$ and $Y$ centralize all of $\n$,   which implies that they belong to $\n$,  which is the desired contradiction because $[X,Y] \neq 0$.
\end{proof}

\subsection{Complex eigenvalues of $\rho(X)$}
\label{ss:eigenvaluesOfRho}
We still assume $N$ abelian and denote by $\rho : \gr \to \gl(\n)$ the representation defined by $\rho(X) = \ad(X)|_{\n}$.  Recall that $G$ denotes $\PO(2,n)$,  $P<G$ the parabolic subgroup introduced in Section \ref{sss:roots}.

\begin{lemma}
\label{lem:sol_like}
Let $X \in \gr$.  If $\rho(X)^{\C}$ has two eigenvalues with distinct, non-zero real parts, then a vector field in $\n$ has a singularity of order $2$.
\end{lemma}

\begin{proof}
Let  $\lambda_1+i\mu_1$ and $\lambda_2 + i\mu_2$, with $\lambda_1 \neq \lambda_2$ and $\lambda_1,\lambda_2 \in \R \setminus\{0\}$ be two such eigenvalues.  Let us choose $X_1,Y_1,X_2,Y_2 \in \n$ all non-zero such that 
\begin{align*}
\begin{cases}
\rho(X) X_1 = \lambda_1 X_1 + \mu_1 Y_1 \\
\rho(X)Y_1 = -\mu_1 X_1 + \lambda_1 Y_1
\end{cases}
\text{ and }
\begin{cases}
\rho(X)X_2 = \lambda_2 X_2 + \mu_2 Y_2 \\
\rho(X)Y_2 = -\mu_2 X_2 + \lambda_2 Y_2.
\end{cases}
\end{align*}
If $\mu_k=0$,  we choose $Y_k=X_k$.  Let $\s$ be the Lie algebra spanned by $X$,  $X_1$,  $X_2$,  $Y_1$, $Y_2$,  and let $S$ denote the integral subgroup of $G$ associated to $\s$.
 
The Zariski closure of $\Ad_{\gr}(e^{tX})$ contains an $\R$-split one-parameter subgroup $\{h^t\}$ such that $h^t(X)=X$,  $h^t(X_1)=e^{\lambda_1 t}X_1$, $h^t(Y_1) = e^{\lambda_1 t}Y_1$, $h^t(X_2) = e^{\lambda_2 t}X_2$, and $h^t(Y_2) = e^{\lambda_2 t}Y_2$.  Also,   if $u_1^t = \Ad_{\gr}(e^{tX_1})$, $u_2^t = \Ad_{\gr}(e^{tX_2})$, $v_1^t = \Ad_{\gr}(e^{tY_1})$, and $v_2^t = \Ad_{\gr}(e^{tY_2})$,  the subgroups $\{u_1^t\}$,$\{u_2^t\}$,$\{v_1^t\}$, and $\{v_2^t\}$ are unipotent,  hence Zariski closed in $\GL(\gr)$.  By the Noetherian property of the discompact radical (Prop.  1.4 of \cite{shalom}), $\bar{S}_d$ contains $\{h^t\}$,  $\{u_1^t\}$,$\{u_2^t\}$,$\{v_1^t\}$, and $\{v_2^t\}$.

Let $K = \{x \in M \ | \ N.x \text{ is isotropic}\}$,  where we consider a fixed point an isotropic orbit.  By Theorem \ref{thm:nilpotent},  $K$ is non-empty, compact and $S$-invariant,  so there exists a point $x \in K$ at which the conclusion of Theorem \ref{thm:bfm} are valid.  For $\hx$ in the fiber of $x$,  we note $S^{\hx} \subset \Ad_{\g}(P)$ the algebraic subgroup obtained in the conclusion of the theorem.  By equivariance,  $S^{\hx.p} = pS^{\hx}p^{-1}$.  
Considering the $\R$-split component of a one-parameter subgroup of $S^{\hx}$ which is sent to $h^t$ (by  Theorem \ref{thm:bfm}),  and changing the point $\hx$ appropriately in the fiber,  we obtain a one-parameter subgroup $\{\Ad_{\g}(e^{tA_X})\} < P^{\hx}$ contained in the $\R$-split Cartan subgroup $A$ described in Section \ref{sss:roots} such that $[A_X,\iota_{\hx}(Y)] = \iota_{\hx}([X,Y])$ for all $Y \in \gr$.  Let $U_1,U_2, V_1,V_2 \in \p$ be  such that $\{\Ad_{\g)}(e^{tU_k})\}$ is sent onto $\{\Ad_{\gr}(e^{tX_k})\}$,  for $k=1,2$,  and $\{\Ad_{\g}(e^{tV_k})\}$ is sent onto $\{\Ad_{\gr}(e^{tY_k})\}$,  for $k=1,2$.  

We show now that $X_1$ and $X_2$ both vanish at $x$.  Since $x \in K$ and $X_1,X_2 \in \n$,  a non-trivial linear combination $\alpha_1 X_1 + \alpha_2 X_2$ vanishes at $x$.  It means that $\alpha_1  \iota_{\hx}(X_1) + \alpha_2 \iota_{\hx}(X_2) \in \p$.  Consequently
\begin{align*}
\alpha_1\lambda_1  \iota_{\hx}(X_1) + \alpha_2\lambda_2 \iota_{\hx}(X_2) = [A_X, \alpha_1  \iota_{\hx}(X_1) + \alpha_2 \iota_{\hx}(X_2) ] \in \p.
\end{align*}
Since $\lambda_1\neq \lambda_2$,  we deduce that $\iota_{\hx}(X_1) \in \p$ or $\iota_{\hx}(X_2) \in \p$.    Let us assume by contradiction that $\iota_{\hx}(X_2) \notin \p$.  Necessarily,  $\iota_{\hx}(X) \notin \p$ because if it did,  then
\begin{align*}
\lambda_2\iota_{\hx}(X_2)+\mu_2  \iota_{\hx}(Y_2)= - [U_2,\iota_{\hx}(X)] \in \p \\
-\mu_2\iota_{\hx}(X_2)+\lambda_2  \iota_{\hx}(Y_2)= - [V_2,\iota_{\hx}(X)] \in \p 
\end{align*}
from which we would derive $\iota_{\hx}(X_2) \in \p$,  a contradiction.  Next,  we see that no element in $\Span\{X,X_2\}$ vanishes at $x$,  because for all $\nu \in \R$,  $[A_X, \iota_{\hx} (\nu X + X_2)] = \lambda_2 \iota_{\hx}(X_2)$,  so $\iota_{\hx} (\nu X + X_2) \notin \p$.  Since $X_2(x)$ is tangent to $N.x$,  it is isotropic in $T_xM$,  and using $[\mu_2V_2-\lambda_2U_2,\iota_{\hx}(X)] = (\lambda_2^2 + \mu_2^2) \iota_{\hx}(X_2)$ and Corollary \ref{cor:virtual_isotropy},  we obtain that 
\begin{align*}
0=g_x(X,[\mu_2Y_2-\lambda_2X_2,  X]) = (\lambda_2^2 + \mu_2^2) g_x(X,X_2),
\end{align*}
and we get that $X(x)$ is orthogonal to $X_2(x)$,  and since they are non-collinear,  we must have $g_x(X,X)>0$.  Since $[A_X,\iota_{\hx}(X)]=0$,  $\iota_{\hx}(X)$ must have a component on $\g_{-\alpha}$ because if not,  then it would have components on both $\g_{-\alpha+\beta}$ and $\g_{-\alpha-\beta}$,  which would force $\alpha(A_X)=\beta(A_X)=0$,  a contradiction because $A_X \neq 0$.  So,  since it has a component on $\g_{-\alpha}$,  we deduce $\alpha(A_X)=0$.  

Consider now the components of $\iota_{\hx}(X_1)$ on the restricted root-spaces.  Since $[A_X,\iota_{\hx}(X_1) ]= \lambda_1 \iota_{\hx}(X_1)$ and $\alpha(A_X)=0$,  we deduce that $\lambda_1 = \beta(A_X)$ and $\iota_{\hx}(X_1) = X_{\beta}^1 + X_{\alpha+\beta}^1$.  As for $\iota_{\hx}(X_2)$,  we deduce similarly that $\lambda_2=-\lambda_1 = -\beta(A_X)$ and that its decomposition is $\iota_{\hx}(X_2) = X_{-\alpha-\beta}^2 + X_{-\beta}^2 + X_{\alpha-\beta}^2$.   Now,
\begin{align*}
0 &  = [\iota_{\hx}(X_1),\iota_{\hx}(X_2)]  \text{ since }  \iota_{\hx}(X_1) \in \p \text{ and using Lemma \ref{lem:iota_crochet}} \\
 & = [X_{\beta}^1,X_{-\alpha-\beta}^2] \ \text{mod}.  \p
 \end{align*}
So, $X_{\beta}^1 = 0$ because $\iota_{\hx}(X_2) \notin \p$.  We deduce similarly $[X_{\alpha+\beta}^1,X_{-\alpha-\beta}^2] = 0$,  from which we finally get $\iota_{\hx}(X_1) = 0$,  the contradiction.

Thus,  we have obtained that both $\iota_{\hx}(X_1)$ and $\iota_{\hx}(X_2)$ belong to $\p$.  We conclude with the following.

\begin{lemma}
Let $H,Y,Z \in \so(2,n)$ be such that $H \in \a$,  $Y,Z \in \p$,  and $[H,Y]=\lambda_1Y$,  $[H,Z] = \lambda_2Z$ and $[Y,Z]=0$.  Then,  $Y \in \p^+$ or $Z \in \p^+$.
\end{lemma}

\begin{proof}
We use the same notations as in Section \ref{sss:roots} for the restricted root-spaces of $\so(2,n)$.  Since $\lambda_1 ,  \lambda_2 \neq 0$,  $Y$ and $Z$ have no component on $\a \oplus \m$.  We denote by $Y = Y_{-\beta} + Y_{\beta} + Y^+$ and $Z = Z_{-\beta}+Z_{\beta}+Z^+$ their decomposition according to $\p = \g_{-\beta}\oplus\g_{\beta}\oplus \a \oplus \m \oplus \p^+$.

If $\beta(H)=0$,  then $Y_{-\beta}=0$ and $Y_{\beta}=0$, (similarly for $Z$),  and the lemma is established.  

We are then reduced to assume that $\beta(H)\neq 0$ and $Y_{\beta} \neq 0$ (the other case being symmetric).  Then,  $\beta(H) = \lambda_1$ and $Y_{-\beta}=0$.  Since $\lambda_1 \neq \lambda_2$,  we get $Z_{\beta} = 0$.   Therefore,  the bracket $[Y,Z]$ decomposes into $[Y,Z] = [Y_{\beta},Z_{-\beta}] + [Y_\beta,Z^+] - [Z_{-\beta},Y^+]$,   the first term belonging to $\a \oplus \m$,  the second and the third to $\p^+$.  It follows therefrom that  $[Y_{\beta},Z_{-\beta}]=0$,  and it follows that $Z_{-\beta}=0$ (see for instance \cite{melnick_pecastaing},  Lemma 3.1).  Finally,  $Y \notin \p^+$ implies $Z \in \p^+$ as expected.
\end{proof}
Applying this fact to $A_X, \iota_{\hx}(X_1),\iota_{\hx}(X_2)$,  we obtain that either $X_1$ or $X_2$ has a singularity of order $2$ at $x$,  concluding the proof.
\end{proof}

\subsection{Essential actions of abelian Lie groups of dimension $n$ and $n-1$}
\label{ss:prop_abelian}

This section is devoted to the proof of the following proposition.

\begin{proposition}
\label{prop:abelian_large}
Let $N$ be an abelian Lie group acting faithfully,  conformally and essentially on a compact Lorentzian manifold $(M^n,g)$,  $n \geq 3$. 
\begin{enumerate}
\item If $\dim N = n$,  then $\n$ contains an element $X$ with a singularity of order $2$.
\item If $N$ is isomorphic to $\R^{n-1}$,  then one of the following is true:
\begin{enumerate}
\item There exists $X \in \n$ with a singularity $x$ around which $X$ is linearizable and locally conjugate to a homothetic flow.
\item There exists $X \in \n$ with a singularity $x$ whose holonomy is non-linear and unipotent.
\item There exists an element $n_0 \in N$ with a fixed point at which its holonomy is of light-like type.
\end{enumerate}
\end{enumerate}
In every case,  an open subset of $M$ is conformally flat.
\end{proposition} 

\begin{remark}
In the second case,  the fact that $N$ has no non-trivial compact subgroup is necessary.  Recall that the conformal group of a Hopf manifold is isomorphic to $\S^1 \times O(1,n-1)$.  Considering a horospherical subgroup $U < O(1,n-1)$,  the abelian group $\S^1 \times U$ is essential,  of dimension $n-1$,  but does not satisfy (a),  (b) or (c).
\end{remark}

We start with the following result.

\begin{lemma}
\label{lem:abelian_so1n}
For all $n \geq 1$,  the maximal dimension of an abelian subalgebra of $\so(1,n+1)$ is $n$.  For $n \geq 3$,  any $n$-dimensional abelian subalgebra of $\so(1,n+1)$ is conjugate to a restricted root-space,  \textit{i.e.} the Lie algebra of a horospherical subgroup in $O(1,n+1)$.
\end{lemma}

\begin{remark}
In $\so(1,3) \simeq \sl_2(\C)$,  an abelian Lie subalgebra of dimension $2$ is either a restricted root-space,  or a (complex) Cartan subalgebra.
\end{remark}

\begin{proof}
We can see a Lie subalgebra $\a \subset \so(1,n+1)$ as a Lie algebra of conformal vector fields of the Möbius $n$-sphere $\S^n$.   According to Remark \ref{rem:riemannian_abelian},  $\a$ acts locally freely on a open-dense subset of the sphere,  giving the upper bound.  

Suppose now $\dim \a = n \geq 3$ and consider the function $\varphi(x) = \sum g_x(X_k,X_k)$ where $(X_1,\ldots,X_n)$ is a basis of $\a$ and $g$ the round metric.  If $\varphi$ was non-vanishing,  then all elements of $\a$ would be Killing vector fields of $g/\varphi$.  But $\Isom(\S^n,g/\varphi)$ would be a compact subgroup of $\PO(1,n+1) = \Conf(\S^n,[g])$,  hence contained in a conjugate of $O(n+1)$.  This would be a contradiction  because any abelian Lie subalgebra of $\so(n+1)$ is contained in a Cartan subalgebra,  and must have dimension at most $\lfloor \frac{n+1}{2}\rfloor$.  We deduce that $\varphi$ vanishes somewhere,  \textit{i.e.} $\a$ fixes a point in $\S^n$.  

So,  we obtain an embedding $\iota : \a \to (\R \oplus \so(n)) \ltimes \R^n$,  the $\R^n$ factor corresponding to the horospherical group at the fixed point of $\S^n$.  The projection on the  $\so(n)$ factor is abelian. Since an hyperplane of $\iota(\a)$ (at least) has no component on $\R$,  and since $n-1 > \Rk_{\C}(\so(n))$,  a non-zero element $Y \in \iota(\a)$ belongs to the $\R^n$ factor.  It follows easily that the embedding has range in $\so(n) \ltimes \R^n$.  Considering $p_1$ the projection on the $\so(n)$ factor and $V = \iota(\a) \cap \R^n= \ker p_1 \circ \iota$,  matrices of $p_1(\iota(\a))$ vanish on $V$,  so belong to $\so(V^{\perp})$.  Since  $p_1(\iota(\a))$ is abelian of dimension $n-m$ in $\so(n-m)$,  we get $m=n$ as expected.
\end{proof}

\subsubsection{Case $\dim N = n$}  According to Proposition \ref{prop:inessential_abelian},  we can choose $x \in M$ such that $\dim N.x \leq 1$.   Let $N_x$ be the stabilizer of $x$.  According to Lemma \ref{lem:iota_crochet},  for any $\hat{x} \in \pi^{-1}(x)$, the map $X \in \n_x \mapsto \omega_{\hx}(X) \in \p\simeq (\R \oplus \so(1,n-1)) \ltimes \R^n$ is a Lie algebra homomorphism.  

If $N_x = N$,  then the composition with the projection on the $\R \oplus \so(1,n-1)$ factor has a non-trivial kernel according to Lemma \ref{lem:abelian_so1n},  and any non-zero $X$ in this kernel satisfies $\omega_{\hx}(X) \in \p^+$ for every $\hx \in \pi^{-1}(x)$.

If $\dim \n_x = n-1$,  then the composition with the projection on the $\so(1,n-1)$ factor has a non-trivial kernel,  and there exists $X \in \n_x$ such that $\omega_{\hx}(X) \in \R.A_{\alpha} \oplus \p^+$,  where $A_{\alpha} \in \a$ is characterized by $\alpha(A_{\alpha})=1$ and $\beta(A_{\alpha})=0$.  Since $\ad(A_{\alpha})$ acts  homothetically on $\g_{-1}$,  the component on $\R.A_{\alpha}$ has to be trivial since $\d_x \phi_X^t$ fixes $X_0(x) \in T_xM$ for any $X_0 \notin \n_x$.   Thus,  $\omega_{\hx}(X) \in \p^+$,  which establishes point (1) in the proposition.  

\subsubsection{Case $N \simeq \R^{n-1}$ with $n \geq 5$} If $N$ has a fixed point $x$,  then what precedes shows that there exists $X \in \n$ such that $\omega_{\hx}(X) = \lambda A_{\alpha} + X_1 \in \R.A_{\alpha} \oplus \p^+$.  If $\lambda \neq 0$,  then $\Ad(p)\omega_{\hx}(X) = \lambda A_{\alpha}$ where $p=e^{\frac{1}{\lambda}X_1}$,  so $\omega_{\hx.p^{-1}}(X) \in \R.A_{\alpha}$,  showing that $X$ is locally conjugate to $\ad(A_{\alpha})|_{\g_{-1}}$ by Proposition \ref{prop:linearization}
and we are in situation (a).  If $\lambda = 0$,  then we in situation (b).  

Consequently,  we can assume that $N$ has no fixed point, and that for all $x \in M$ such that $N.x$ is $1$-dimensional and light-like,  the embedding $\iota_{\hx} : X \in \n_x \mapsto \omega_{\hx}(X) \in \p$ is such that $p \circ \iota_{\hx}$ is injective,  where $p : (\R \oplus \so(1,n-1)) \ltimes \R^n \to \so(1,n-1)$ denotes the projection (if not,  we are in situation (b)).  Furthermore,  we can assume that for all $X \in \n_x$,  if $\iota_{\hx}(X) \in \p$ is unipotent,  then it is conjugate in $P$ to an element of $\g_0$.

Recall that $K$ refers to the compact subset of $M$ where $N$-orbits are totally isotropic.  We note $\n_x$ the Lie algebra of the stabilizer in $N$ of a point $x$.  The following Lemma will be reused later.

\begin{lemma}
\label{lem:holonomy_gbeta}
Let $x \in K$.  If $n \geq 5$,  and if we are note in cases (a) and (b) of Proposition \ref{prop:abelian_large},  then there exists $\hx \in \pi^{-1}(x)$ such that $\iota_{\hx}(\n_x) = \g_{\beta}$.  
\end{lemma}

\begin{proof}
If $p : \p \to \so(1,n-1)$ denotes the projection,  since $\dim \n_x = n-2$,  Lemma \ref{lem:abelian_so1n} implies that $p \circ \iota_{\hx}(\n_x) \subset \so(1,n-1)$ is a restricted root-space.  So,  replacing $\hx$ by some $\hx .  g_0$,  with $g_0 \in G_0$,  we can assume $p \circ \iota_{\hx}(\n_x) = \g_{\beta}$.

In fact,  no element of $\iota_{\hx}(\n_x)$ has a component on the grading element $A_{\alpha}$.  Indeed,  every one-parameter subgroup $\{\phi_X^t\} \subset N_x$ fixes an isotropic vector in $T_xM$ which corresponds to the tangent line of the $N$-orbit of $x$.  Given the form of $\iota_{\hx}(X)$,   $\d_x\phi_X^t = e^{\lambda t} u^t$,  where $\{u^t\}$ is a unipotent one-parameter subgroup of $O(T_xM)$ and $t$ the component of $\iota_{\hx}(X)$ on $A_{\alpha}$.  Hence,  $\lambda =0$.

Thus,  $\iota_{\hx}(\n_x)$ is included in $\g_{\beta} \oplus \g_{\alpha-\beta} \oplus \g_{\alpha} \oplus \g_{\alpha+\beta}$.  Next,  we show that it is in fact in $\g_{\beta}  \oplus \g_{\alpha} \oplus \g_{\alpha+\beta}$.  Let $\iota_{\hx}(X) = X_{\beta}+ X_{\alpha-\beta} + X_{\alpha} + X_{\alpha+\beta}$ be the decomposition for some element $X \in \n_x$.  Choose $Y \in \n_x$ non-zero and such that $Y$ has no component on $\g_{\alpha-\beta}$ (recall that $\dim \g_{\alpha-\beta} = 1$),  and let $Y_{\beta}+ Y_{\alpha} + Y_{\alpha+\beta}$ be the components of $\iota_{\hx}(Y)$.  Expressing that the bracket is zero,  we get $[X_{\alpha-\beta},Y_{\beta}] = 0$,  and since $Y_{\beta} \neq 0$,  we obtain $X_{\alpha-\beta}=0$.

Therefore,  every $X \in \n_x$ has a unipotent holonomy,  which is then conjugate to an element of $\g_0$ by assumption.   The following shows that the conjugacy is realized by a same element of $P^+$,  \textit{i.e.} that we have a simultaneous linearization of the isotropy.  This is a slight variation of Proposition 5.2 of \cite{melnick_pecastaing}.  Since the hypothesis are different,  we give a proof in this particular setting.

\begin{lemma}
\begin{enumerate}
\item An element $X_{\beta}+ X_{\alpha} + X_{\alpha+\beta} \in \p$ with $X_{\beta} \neq 0$ is conjugate to an element of $\g_0$ if and only if there exists $Z \in \g_{\alpha-\beta}$ such that $X_{\alpha} = [Z,X_{\beta}]$.  
\item If $\h \subset \g_{\beta}\oplus \g_{\alpha} \oplus \g_{\alpha+\beta}$ is a linear subspace all of whose elements are conjugate to an element of $\g_0$,  then there exists $p \in \exp(\g_{\alpha-\beta} \oplus \g_{\alpha})$ such that $\Ad(p)\h \subset \g_{\beta}$.
\end{enumerate}
\end{lemma}

\begin{proof}
We identify $P$ as the subgroup $\Ad(G_0)|_{\p^+} \ltimes \p^+$ of the affine group of $\p^+$,  and its Lie algebra correspondingly.  A one-parameter subgroup $\{e^{tX}\} \subset P$ is conjugate to a one-parameter subgroup of $G_0$ if and only if its affine action on $\p^+$ has a fixed point,  or equivalently,  writing $X = (X_0, X_1) \in \g_0 \ltimes \p^+$,  if there exists $Y_1 \in \p^+$ such that $X_1 = [X_0,Y_1]$ and then $\Ad(e^{Y_1}) X = X - [X_0,Y_1] = X_0$.

Recall that if $Z \in \g_{\alpha-\beta}$ is non-zero,  then $\ad(Z)|_{\g_{\beta}} : \g_{\beta} \to \g_{\alpha}$ is a linear isomorphism and that the bracket $\g_{\alpha} \times \g_{\beta} \to \g_{\alpha+\beta}$ is non-degenerate in the sense that it identifies $\g_{\alpha}$ with $(\g_{\beta})^*$ and conversely.

For the first point,  if $X_{\beta}+ X_{\alpha} + X_{\alpha+\beta}$ is conjugate to an element of $\g_0$,  then there exists $Y_1 = Y_{\alpha-\beta} + Y_{\alpha} + Y_{\alpha + \beta}$ such that $X_{\alpha} + X_{\alpha+\beta} = [X_{\beta},Y_1] = [X_{\beta},Y_{\alpha-\beta} ] + [X_{\beta},Y_{\alpha}]$,  so $Z:= -Y_{\alpha-\beta}$ is convenient.  Conversely,  if there exists $Z \in \g_{\alpha-\beta}$ such that $X_{\alpha} = [Z,X_{\beta}]$,  then if $Z' \in \g_{\alpha}$ is such that $[Z',X_{\beta}] = X_{\alpha+\beta}$ ,  then $[X_{\beta},-Z-Z'] = X_{\alpha}+X_{\alpha+\beta}$,  showing that $\Ad(e^{-Z-Z'}) (X_{\beta}+ X_{\alpha} + X_{\alpha+\beta}) = X_{\beta}$.

For the second point,  the linear projection to $\g_{\beta}$ is injective in restriction to $\h$.   We deduce the existence of a linear subspace $V \subset \g_{\beta}$ and linear maps $\varphi_{\alpha} : V \to \g_{\alpha}$ and $\varphi_{\alpha + \beta} : V \to \g_{\alpha+\beta}$ such that $\h = \{X + \varphi_{\alpha}(X) + \varphi_{\alpha+\beta}(X),  \ X \in V\}$. Our hypothesis is that for a given non-zero $Z_{\alpha-\beta} \in \g_{\alpha-\beta}$,  $\varphi_{\alpha}(X)$ is collinear to $[Z_{\alpha-\beta},X]$ for every $X \in V$.  Let $f : \g_{\alpha} \to \g_{\beta}$ be the  inverse of $(\ad Z_{\alpha-\beta})|_{\g_{\beta}}$.  Then,  $f \circ \varphi_{\alpha} : V \to V$ is a linear map fixing all the lines,  hence an homothety.  Replacing $Z_{\alpha-\beta}$ by $\lambda Z_{\alpha-\beta}$ if necessary,  we must have $\varphi_{\alpha}(X) = [Z_{\alpha-\beta},X]$ for all $X \in V$.  Since the bracket $\g_{\alpha} \times \g_{\beta} \to \g_{\alpha+\beta}$ is non-degenerate,  there exists $Z_{\alpha}\in \g_{\alpha}$ such that $\varphi_{\alpha+\beta}(X) = [Z_{\alpha},X]$ for all $X \in V$.  Finally,  we obtain $\Ad(e^{-Z_{\alpha-\beta}-Z_{\alpha}})\h =V \subset \g_{\beta}$ as announced.
\end{proof}
Replacing now $\hx$ by $\hx.p^{-1}$,  we obtain $\iota_{\hx}(\n_x) \subset \g_{\beta}$,  which concludes the proof of the Lemma since $\dim \n_x = \dim \g_{\beta} = n-2$.
\end{proof}

Combined with Proposition \ref{prop:linearization},  Lemma \ref{lem:holonomy_gbeta} implies that the action of the identity component of the stabilizer $(N_x)_0$ is locally linearizable at the neighborhood of $x$,  and explicitly the local conjugacy is given by
\begin{align*}
\exp_{\hx} : \mathcal{U} \subset \g_{-1} \to U \subset M,
\end{align*}
where $\g_{-1} = \g_{-\alpha+\beta}\oplus \g_{-\alpha} \oplus \g_{-\alpha-\beta}$,  $\mathcal{U}$ is a neighborhood of $0$ and $U$ a neighborhood of $x$.  For all $X \in \n_x$,  we have $(\exp_{\hx})^* X = \ad(\omega_{\hx}(X))|_{\g_{-1}}$ over $\mathcal{U}$.  Therefore, because $\g_{-\alpha + \beta}$ is the subspace of $\g_{-1}$ centralized by $\g_{\beta}$,  the set of fixed points of $(N_x)_0$ in $U$ is a $1$-dimensional submanifold, parametrized by $\pi(\exp_{\hx}(tX_{-\alpha+\beta}))$ for $t$ in a neighborhood of $0$.  Necessarily,  this $1$-dimensional fixed points set coincides locally with the $N$-orbit of $x$.  We get the following consequence: 

\begin{lemma}
\label{lem:Norbit_closed}
Any $1$-dimensional,  light-like $N$-orbit is closed.
\end{lemma}

\begin{proof}
Let $x$ be a point with a $1$-dimensional,  light-like orbit.  Let $y \in \bar{N.x}$ be a point in the orbit closure.  Since $N$ is abelian,  $N_x \subset N_y$,  so that $\n_x =\n_y$ since we assumed that $N$ has no fixed point.  Let $U$ be a neighborhood of $y$ over which $(N_y)_0$ is locally linearizable.   We have observed that $\{z \in U \ | \ \forall g \in (N_y)_0,  \ g.z =z\} = N.y  \cap U$.  Let now $g_0 \in N$ be such that $g_0.x \in U$.  Then,  $g_0.x$ is fixed by $N_x$,  which contains $(N_y)_0$,  so $g_0.x \in N.y$.  Therefore,  $y \in N.x$,  showing that $N.x$ is closed as we claimed.
\end{proof}

Therefore,  for every $X \in \n$ and every $x \in K$,  if $X(x) \neq 0$, then there exists $t_0 > 0$ such that $e^{t_0X} \in N_x$.  In general,  it is possible that $e^{t_0X} \in (N_x)_0$ (for instance if $\phi_X^t$ is a free $\S^1$-action).  This is not possible in our situation since $N$ is isomorphic to $\R^{n-1}$.

\begin{lemma}
The centralizer $\mathcal{Z}_{O(1,n-1)}(\g_{\beta})$ is $\mathcal{Z}_{O(1,n-1) } \times \exp(\g_{\beta})$.
\end{lemma}

\begin{proof}
Consider the standard conformal action of $O(1,n-1)$ on the round sphere $\S^{n-2}$.  Let $x \in \S^{n-2}$ be the (unique) point fixed by all of $\g_{\beta}$. Then,  for $g$ in the centralizer,  we must have $g.x=x$ by uniqueness.  Therefore,  $g \in Q < O(1,n-1)$,  the  parabolic subgroup fixing $x$.  The group $Q$ decomposes into $Q = (\R^* \times O(n-2)) \ltimes \R^{n-2}$,  where the $\R^{n-2}$ factor parametrizes $\exp(\g_{\beta}) < Q$.   As the latter clearly centralizes $\g_{\beta}$,  we can assume that $g$ has no component on the $\R^{n-2}$ factor.  Let $(a,A) \in \R^* \times O(n-2)$ be corresponding to the other factors of $g$, then  the adjoint action of $g$  on some $X_{\beta}$ in the root-space, which is parametrized by $u \in \R^{n-2}$,  is $a^{-1} Au$.  We get $Au = au$,  for all $u \in \R^{n-2}$,  showing that $A$ is a scalar matrix,  and finally $g = \pm \id$.
\end{proof}

\begin{corollary}
If $x \in K$,  then there exists $X \in \n$,  $\hx \in \pi^{-1}(x)$ and $t_0 > 0$ such that $X(x) \neq 0$,  $n_0 := e^{t_0X} \in N_x$,  and $\hol^{\hx}(n_0)$ is a light-like translation in $P^+$.  
\end{corollary}

\begin{proof}
Let $X \in \n$ be such that $X(x) \neq 0$ and let $t_0>0$ be such that $n_0 = e^{t_0X} \in N_x$.  Let $\hx \in \pi^{-1}(x)$ be such that $\omega_{\hx}(\n_x) = \g_{\beta}$. 

Recall that the map $\hol^{\hx} : f \in \Conf(M,g)_x \mapsto \hol^{\hx}(f) \in P$  is an injective Lie group homomorphism,  where $\hol^{\hx}$ refers to the unique $p \in P$ such that $\hat{f}(\hx) = \hx.p^{-1}$.  

So,  $\hol^{\hx} (n_0) \in P$ is an element centralizing $\hol^{\hx}((N_x)_0) = \exp(\g_{\beta})$.  Since $\hol^{\hx}$ is injective and $n_0 \notin (N_x)_0$,  $\hol^{\hx}(n_0) \notin \exp(\g_{\beta})$.   Using the decomposition $P = G_0 \ltimes P^+$,  and replacing $t_0$ by $2t_0$ if necessary,  we obtain that the $G_0$-part of $\hol^{\hx} (n_0)$ belongs to $\{A^t\} \times \exp(\g_{\beta})$,  say $(e^{s_0 A},e^{X_{\beta}})$ for $X_{\beta} \in \g_{\beta}$.  Thus,  picking $Y \in \n_x$ such that $\hol^{\hx}(e^{t_0}Y) = e^{X_{\beta}}$,  and replacing $X$ by $X-Y$,  we may assume that $\hol^{\hx}(n_0) \in \{A^t\} \ltimes P^+$ and is a non-trivial element.  Necessarily,  the factor on $A^t$ must be the identity because if not,  it would mean that the differential $\d_x n_0$ is a non-trivial homothety,  whereas it fixes pointwise the light-like periodic orbit $N.x$.  Finally,  $\hol^{\hx}(n_0)$ is a non-trivial element of $P^+$ which centralizes $\g_{\beta}$,  so we conclude $\hol^{\hx}(n_0) \in \exp(\g_{\alpha + \beta})$.
\end{proof}

This concludes the proof of Proposition \ref{prop:abelian_large} in the case $n \geq 5$.

\subsubsection{Case of $\R^{n-1}$,  $n=3$}

In this situation,  $\n_x$ is a line.  If $\{\phi^t\}$ is a one-parameter subgroup of $N_x$,  then $\d_x \phi^t$ fixes an isotropic vector in $T_xM$ given by the orbit $N.x$.  As a consequence,  $\d_x \phi^t = e^{\lambda t} g^t$,  where $\{g^t\} \subset \SO(T_xM,g_x) \simeq \SO(1,2)$ is in the parabolic subgroup fixing this isotropic line.  It implies that either $\lambda=0$ and $\d_x \phi^t = g^t$ is unipotent,  or $\{g^t\}$ is hyperbolic and $\lambda \neq 0$.  

\begin{lemma}
\label{lem:casen=3}
For any $x \in K$,  the identity component $(N_x)_0 = \{\phi_Y^t\}_{t \in \R}$ is locally linearizable near $x$,  and $Y$ is locally conjugate to one of the two linear vector fields on $\R^3$
\begin{align*}
X_h = 
\left |
\begin{array}{l}
0 \\
x_2 \\
2x_3
\end{array}
\right.
\text{ or }
X_u =
\left |
\begin{array}{l}
x_2 \\
-x_3 \\
0
\end{array}
\right.
.
\end{align*}
In both cases,  there exists a neighborhood $U$ of $x$ such that $\{y \in U \ | \ Y(y)=0\} = N.x \cap U$ is a one-dimensional submanifold of $U$.
\end{lemma}

\begin{proof}
If $\d_x \phi_Y^t$ is unipotent,  then the proof of the case $n \geq 5$ can be directly applied.  

So,  we assume that $\d_x \phi_Y^t$ is hyperbolic.  Up to conjugacy in $P$,  we have $\iota_{\hx}(Y) = Y_0 + Y_1$ where
\begin{align*}
Y_0 = 
\begin{pmatrix}
1 & & & & \\
 & 1 & & & \\
 & & 0 & & \\
 & & & -1 & \\
 & & & & -1
\end{pmatrix}
\end{align*}
and $Y_1 \in \p^+$.  After conjugacy by an element of $\exp(\g_{\alpha}\oplus \g_{\alpha+\beta})$,  we can furthermore assume $Y_1 = Y_{\alpha-\beta} \in \g_{\alpha-\beta}$.  Now let $X \in \n$ be such that $X(x) \neq 0$.  Since $[X,Y]=0$ and $Y(x) = 0$,  $[\iota_{\hx}(X),\iota_{\hx}(Y)]  = 0$.  It follows that the  $\g_{-1}$ component of $\iota_{\hx}(X)$ is in $\g_{-\alpha+\beta}$.  Let $\iota_{\hx}(X) = X_{-\alpha+\beta}+X_0+X_1$ be the decomposition according to the grading of $\so(2,3)$.  Expressing that the bracket $[\iota_{\hx}(X),\iota_{\hx}(Y)]=0$,  we obtain $[X_{-\alpha+\beta},Y_{\alpha-\beta}]+[X_0,Y_0]=0$.  This forces $Y_{\alpha-\beta}=0$ because $[\g_0,\g_0]$ is contained in the $\so(1,2)$ factor of $\g_0$ whereas $[\g_{\alpha-\beta},\g_{-\alpha+\beta}]$ is not.  By Proposition \ref{prop:linearization},  $Y$ is locally linearizable near $x$ and conjugate to the linear vector field $\ad(Y_0)|_{\g_{-1}}$.  The latter fixes only the line $\g_{-\alpha+\beta}$,  so the fixed points of $Y$ near $x$ form a $1$-dimensional submanifold near $x$,  which necessarily locally coincides with $N.y$, as claimed.
\end{proof}

It follows,  as in Lemma \ref{lem:Norbit_closed} in the case $n \geq 5$,  that for all $x \in K$,  $N.x$ is closed.   Let $X \in \n$ be such that $X(x) \neq 0$ and let $t_0>0$ be such that $n_0 = e^{t_0X} \in N_x$.  Since $N \simeq \R^2$,  we have $n_0 \notin (N_x)_0$.  

If $(N_x)_0$ has a unipotent linear action near $x$,  then the end of the proof of the case $n \geq 5$ applies.  Let us assume then that we are in the hyperbolic case.  The proof of Lemma \ref{lem:casen=3} shows that there exists $\hx \in \pi^{-1}(x)$ such that $\iota_{\hx}(Y)
 = \diag(1,1,0,-1,-1)$.  Hence,  $\hol^{\hx}(n_0)$ is an element of $P$ centralizing $\iota_{\hx}(Y)$.  It follows that there exists $s_0 \in \R$ such that $\hol^{\hx}(n_0.\phi_Y^{s_0}) \in \exp(\g_{\alpha-\beta})$,  which concludes since $n_0 \notin (N_x)_0$.

\subsubsection{Case of $\R^{n-1}$,  $n=4$}

When $n=4$,  by Lemma \ref{lem:abelian_so1n},  considering the projection $p : \p \to \so(1,3)$ as before,  we get that for all $x \in K$ and $\hx \in \pi^{-1}(x)$,  $p \circ \iota_{\hx}(\n_x)$ is either a restricted root-space,  or a Cartan subalgebra.

As in the previous cases,  we obtain that near every point $x \in K$,  there exists $Y \in \n_x$ which is locally linearizable and whose fixed points locally coincide with $N.x$ (in the unipotent case,  the proof is the same as in the case $n \geq 5$,  in the other case,  we pick $Y \in \n_x$ such that $p \circ \iota_{\hx}(Y)$ is an $\R$-split element in the Cartan subalgebra,  and the proof of the case $n=3$ applies).

It follows similarly that $N.x$ is closed for every $x \in K$.  Let  $t_0>0$ be such that $n_0 = e^{t_0}X \in N_x$.  If the isotropy representation of $(N_x)_0$ at $x$ is unipotent,  then the same approach as in the case $n \geq 5$ implies the existence of some $Y_0 \in \n_x$ and $\hx \in \pi^{-1}(x)$ such that $\hol^{\hx}(n_0e^{Y_0}) \in \exp(\g_{\alpha + \beta})$ (and non-trivial).  If the isotropy representation sends $(N_x)_0$ onto a Cartan subgroup of $\SO(T_xM,g_x)$,  then the same proof as in the case $n=3$ shows that if $Y \in \n_x$ is sent onto the $\R$-split element of the Cartan subalgebra and $Z \in \n_x$ to the imaginary part,  then there exists $\hx \in \pi^{-1}(x)$ such that
\begin{align*}
\iota_{\hx}(Y) = 
\begin{pmatrix}
1 & & & & & \\
& 1 & & & & \\
& & 0 & & & \\
& & & 0 & & \\
& & & & -1 & \\
& & & & & -1
\end{pmatrix}
\text{ and }
\iota_{\hx}(Z) = 
\begin{pmatrix}
0 & & & & & \\
& 0 & & & & \\
& & 0 & -1 & & \\
& & 1  & 0 & & \\
& & & & 0 & \\
& & & & & 0
\end{pmatrix}.
\end{align*}
Expressing that $\hol^{\hx}(n_0)$ centralizes these elements,  we obtain that 
\begin{align*}
\hol^{\hx}(n_0) =
\begin{pmatrix}
\lambda & x_0 & & &  \\
0 & \mu & & & \\
 & & R & & \\
 & & & \mu^{-1} & -x_0\\
 & & & & \lambda^{-1}
\end{pmatrix},
\end{align*}
with $R  \in O(2)$,  $\lambda,\mu \in \R^*$.
Since $\d_x n_0$ fixes the same isotropic vector in $T_xM$ as $Y$,  it follows that $\lambda=\mu$.  So,  replacing $n_0$ by its square,  and translating it by some $e^{s_0Y+s_1Z}$,  we obtain
\begin{align*}
\hol^{\hx}(n_0) =
\begin{pmatrix}
1 & x_0 & & &  \\
0 & 1 & & & \\
 & & \id & & \\
 & & &1 & -x_0\\
 & & & & 1
\end{pmatrix} \in \exp(\g_{\alpha-\beta}),
\end{align*}
as expected.

\section{Proof of Theorem \ref{thm:embedding}: embedding of the radical into $\so(2,n)$}
\label{s:proof_embedding}

We assume in all this section that no non-empty open subset of $(M^n,g)$ is conformally flat.	We denote by $R$ a connected,  solvable Lie subgroup of $\Conf(M,g)$,  which is assumed to be essential.  By Theorem \ref{thm:inessential_nilradical} and Theorem \ref{thm:nilpotent},  $R$ has abelian nilradical $N$ and $N$ is essential.  By Proposition \ref{prop:abelian_large},  $\dim N \leq n-1$ and if $\dim N =n-1$,  then the maximal torus of $N$ is non-trivial.  By Propositions \ref{prop:split_exact_sequence} and \ref{prop:rho_semisimple},  there exists $\a \subset \gr$ abelian and $\rho : \a \to \gl(\n)$ a faithful representation such that for all $X\in \a$,  $\rho(X)$ is semi-simple and $\gr \simeq \a \ltimes_{\rho} \n$.  Let $\{\alpha_1,\bar{\alpha_1},\ldots,\alpha_r,\bar{\alpha_r}\} \subset (\a^*)^{\C}$ be the complex weights of $\rho$.  According to Lemma \ref{lem:sol_like},  there exists real linear forms $\lambda,\mu_1,\ldots,\mu_r \in \a^*$ and $s \in \{0,\ldots,r\}$ such that $\lambda \neq 0$ and 
\begin{itemize}
\item $\alpha_k = \lambda + i \mu_k$ for $1 \leq k \leq s$ ;
\item $\alpha_k = i\mu_k$ for $s+1 \leq k \leq r$,
\end{itemize}
the case where all weights are purely imaginary corresponding to $s=0$.  

We set first some notations.  Let us take the convention that if there exists $k \in \{1,\ldots,s\}$ such that $\mu_k = 0$,  then it is $k=1$,  and if there exists $k \in \{s+1,\ldots,r\}$ such that $\mu_k = 0$,  then it is $k=s+1$.   Let $\n^{\C} =\bigoplus_k (\n^{\C})_{\alpha_k}$ denote the complex weight-space decomposition.  Remark that $\cap_k \ker \alpha_k = \{0\}$  since $\rho$ is faithful (the centralizer of $\n$ in $\gr$ coincides with $\n$).  At most two weight-spaces are real: $(\n^{\C})_{\alpha_1}$ if $\mu_1=0$ and $(\n^{\C})_{\alpha_{s+1}}$ if $\mu_{s+1}=0$,  the latter corresponding to the center of $\gr$.  For a non-real weight $\alpha_k$,  we define the real even-dimensional subspace 
\begin{align*}
\n^k = \{X+\bar{X}, \ X \in (\n^{\C})_{\alpha_k}\} \oplus \{i(X-\bar{X}), \ X \in (\n^{\C})_{\alpha_k}\} \subset \n,
\end{align*}
which verifies $\n^k \oplus i\n^k = (\n^{\C})_{\alpha_k} \oplus \bar{(\n^{\C})_{\alpha_k}}$.  Finally,  let 
\begin{align*}
\n_1 = \bigoplus_{\alpha_k \notin \a^* \cup \,  i \a^*} \n_k\  ,  \quad  
\n_2 = \bigoplus_{ 
 \alpha_k \in i \a^* \setminus \{0\}} \n_k.
\end{align*}
Hence,  we have the decomposition
\begin{align*}
\n = \n_1 \oplus \n_2 \oplus \n_{\lambda} \oplus \z(\gr),
\end{align*}
where $\z(\gr)$ denotes the center of $\gr$ and $\n_{\lambda} = \{X \in \n \ | \ \forall H \in \a, \ [H,X]  = \lambda(H)X\}$.   

\subsection{Case $\dim N \leq n-2$} In this situation,  we have the following explicit embedding of any such solvable Lie algebra $\gr$ into $\so(2,n)$,  which sends the nilradical into $\g_{\alpha} \oplus \g_{\beta}$.  Note that $\dim \g_{\alpha} = \dim \g_{\beta} = n-2$.   To an element $(X,Y)$ with $X \in \a$ and $Y \in \n$ with decomposition $Y = Y_1+Y_2+Y_{\lambda}+Y_{\z}$,  we associate the matrix of $\so(2,n)$:

\begin{align*}
\left (
\begin{array}{cc|ccccc|cc}
\lambda(X) &  &\quad Y_1 \quad  &  \quad Y_{\lambda} \quad  & \quad 0 & \quad 0  & \quad 0 \quad &\\
                     &  0 &      \quad 0        &  \quad 0 & \quad Y_2 \quad & \quad Y_{\z} \quad & \quad 0 &&\\
\hline                     
                     &    & D_1(X) &  & & & &0&- \ \! \!^{t} Y_1 \\
                     &    &               &  \mathbf{0}  & & & & 0&- \ \! \!^{t} Y_{\lambda} \\
                     &   &                     &            & D_2(X)  & & & - \ \! \!^{t} Y_2 &0\\
                     &   &                     &            &                    & \mathbf{0}  &  & - \ \! \!^{t} Y_{\z} &0 \\
                     &   &                     &            &                    &                       & \mathbf{0} & 0&0 \\
\hline                     
                     &   &                     &    &                &                &              &0 &          \\          
                     &   &                     &    &                &                &              &   & -\lambda(X)          
\end{array}
\right )
\end{align*}
where 
\begin{align*}
& D_1(X) =
\begin{pmatrix}
0 & -\mu_1(X) & & &  \\
\mu_1(X) &0 & & &  \\
  &   & \ddots & & \\
  &   &              & 0 & -\mu_{s}(X)  \\
  &   &              & \mu_s(X) & 0  \\
\end{pmatrix}
\\
\text{and } &
D_2(X) =
\begin{pmatrix}
0 & -\mu_{s+1}(X) & & &  \\
\mu_{s+1}X) &0 & & &  \\
  &   & \ddots & & \\
  &   &              & 0 & -\mu_{r}(X)  \\
  &   &              & \mu_r(X) & 0  \\
\end{pmatrix}.
\end{align*}

\subsection{Case $\dim N = n-1$}

Recall that we can assume that $N$ has a non-trivial maximal torus by Proposition \ref{prop:abelian_large} and that $N$ acts without fixed points.  We still denote by $K$ the compact subset of points admitting a one-dimensional light-like $N$-orbit.  

\begin{lemma}
\label{lem:dimN=n-1}
If $\n_1 \neq 0$,  then $\n_2 = 0$ and $\dim \z(\gr) \leq 1$.
\end{lemma}

\begin{proof}
The result is immediate if $n<5$.  Let us then assume $n \geq 5$.  Let $X \in \a$ be such that $\lambda(X)=1$.  We apply Theorem \ref{thm:bfm} to the pair $(G,S)$ where $S = \{e^{tX}\}_{t\in \R}$.  Considering a finite measure supported on $K$,  we deduce the existence of $x\in K$ at which the conclusions of this theorem are true.  Similarly as before,  the Zariski closure of $\Ad(S)$ in $\GL(\gr)$ contains an $\R$-split one-parameter subgroup $\{h^t\}$ which acts homothetically on $\n_1 \oplus \n_{\lambda}$ and trivially on $\n_2 \oplus \z(\gr)$.  

Necessarily,  $\{h^t\}$ is contained in the discompact radical of this Zariski closure.  Choose now $\hx \in \pi^{-1}(x)$ such that $\iota_{\hx}(\n_x) = \g_{\beta}$.  Considering the Jordan decomposition of a one parameter subgroup of $\Ad_{\so(2,n)}(P)$ which is sent onto $\{h^t\}$ by the algebraic homomorphism provided by Theorem \ref{thm:bfm},  we get the existence of an $\R$-split one parameter subgroup $\{p^t\} < P$ such that for all $X \in \gr$,  $\Ad(p^{t})\iota_{\hx}(X) = \iota_{\hx}(h^t(X))$.  Let $\chech{X} \in \p$ be the generator of $\{p^t\}$. Then,  we have:
\begin{itemize}
\item $[\chech{X},\iota_{\hx}(Y)] =  \iota_{\hx}(Y)$ for all $Y \in \n_1 \oplus \n_{\lambda}$ ;
\item $[\chech{X},\iota_{\hx}(Y)] =  0$ for all $Y \in \n_2 \oplus \z(\gr)$.
\end{itemize}
Recall that $\p = (\R \oplus \so(1,n-1)) \ltimes \R^n$.  Consider $X_0$ the $\so(1,n-1)$ component of $\chech{X}$.  Then,  $\ad(X_0)$ is still $\R$-split.  It is moreover non-zero because if it was,  it would imply (in particular) that $\chech{X}$ centralizes $\g_{\beta}$,  but this is excluded because $\dim \n_1 \geq 2$,  so it intersects $\n_x$ and there exists $Y \in \n_1$ non-zero such that $\iota_{\hx}(Y) \in \g_{\beta}$ and $[\chech{X},\iota_{\hx}(Y)] = [X_0, \iota_{\hx}(Y)] =  \iota_{\hx}(Y)$.  So,  $X_0$ generates an $\R$-split Cartan subgroup of $\so(1,n-1)$.  Consequently,  the centralizer of $X_0$ in $\so(1,n-1)$ is isomorphic to $\R.X_0 \oplus \so(n-2)$,  and an element of $\so(1,n-1)$ centralizing $X_0$ is semi-simple,  so it cannot belong to a restricted root-space.  

If  $\n_2$ was non-zero,  then,  being at least $2$-dimensional,  it would intersect the hyperplane $\n_x$.  But for $Y \in \n_x \cap \n_2$, we have $\iota_{\hx}(Y) \in \g_{\beta}$ and $[\chech{X},\iota_{\hx}(Y)] = [X_0,\iota_{\hx}(Y)] = 0$,  a contradiction.
\end{proof}

\begin{lemma}
\label{lem:centreNonNul}
$\z(\gr) \neq 0$.
\end{lemma}

\begin{proof}
Let $T \simeq \T^k$ denote the maximal torus of the abelian Lie group $N$.  Then,  $T$ is normal in $R$ and since $\Aut(T)$ is discrete,  $T$ is in fact central in $R$ and,  by assumption, $\t \neq 0$.  So, $\z(\gr) \neq 0$.
\end{proof}

We can now conclude by distinguishing several cases.  

\subsubsection{$\n_{\lambda}\neq 0$}

In this situation,  we distinguish a direction in $\n_{\lambda}$ (noted $Y'_{\lambda}$) that we embed into $\g_{\alpha-\beta}$:

\begin{align*}
\left (
\begin{array}{cc|cccc|cc}
\lambda(X) & Y_{\lambda}' & \quad Y_{\lambda}'' \quad  & \quad Y_1 & \quad 0 &  \quad 0 \quad  \quad & & \\
                     &  0 &      \quad 0        &  \quad 0 &  \quad Y_2 \quad  & \quad Y_{\z} \quad   &  & \\
\hline                     
                     &    & \mathbf{0} &  & &   & &- \ \! \!^{t} Y_{\lambda}'' \\
                     &   &  & D_1(X) & & & & - \! \!^{t} Y_1 \\
                     &    &              & &  D_2(X) &   & - \ \! \!^{t} Y_2&\\
                     &   &                        &      &   &  \mathbf{0} &    - \ \! \!^{t} Y_{\z} &  \\
\hline                     
                     &   &                     &                    &                              & & 0 &   -Y_{\lambda}'      \\          
                     &   &                     &                  &                              &   & & -\lambda(X)          
\end{array}
\right )
\end{align*}

\subsubsection{Case $\n_{\lambda} = 0$ and $\n_1 \neq 0$}

Recall that by Lemma \ref{lem:dimN=n-1},  we have $\n_2 = 0$ and the center is contained in a line,  so $\z(\gr))$ is a line by Lemma \ref{lem:centreNonNul}, that we embed into $\g_{\alpha-\beta}$ as follows

\begin{align*}
\left (
\begin{array}{cc|c|cc}
                \lambda(X)     &      \quad Y_{\z}   &   \quad Y_1 \quad   & & \\
                     & \lambda(X)  & &  & \\
\hline                     
                     &    &                D_1(X)    & & - \ \! \!^{t} Y_1\\
\hline                     
                     &                        &                                               & -\lambda(X) & -Y_{\z}    \\          
                     &                    &                                               &     & -\lambda(X)         
\end{array}
\right )
\end{align*}

\subsubsection{Case $\n_{\lambda} = 0$ and $\n_1=0$}
In this last configuration,  we have the embedding
\begin{align*}
\left (
\begin{array}{cc|cc|cc}
                0     &  Y_{\z}' &      \quad Y_{\z}''    &   \quad Y_2 \quad   & & \\
                     & 0  & &  & \\
\hline                     
                     &    & \mathbf{0} &  & & - \ \! \!^{t} Y_{\z}'' \\
                     &    &               &  D_2(X)    & & - \ \! \!^{t} Y_2\\
\hline                     
                     &   &                     &                                               & 0 & -Y_{\z}'    \\          
                     &   &                     &                                               &     & 0         
\end{array}
\right )
\end{align*}

\bibliographystyle{alpha}
\bibliography{bibli}

\end{document}